\documentclass[preprint]{imsart}
\pubyear{2023}
\volume{TBA}
\issue{TBA}
\firstpage{1}
\lastpage{1}

\usepackage{amsthm}
\usepackage{amsmath}
\usepackage{amsfonts}

\usepackage{natbib}
\usepackage[colorlinks,citecolor=blue,urlcolor=blue,filecolor=blue,backref=page]{hyperref}
\usepackage{graphicx}
\usepackage{float} 
\usepackage{booktabs} 
\usepackage{graphicx} 
\usepackage[margin=1cm]{caption} 
\usepackage{mathtools}
\usepackage{subfigure}
\usepackage{pgfplots}
\usepackage{tikz}
\newlength\figureheight
\newlength\figurewidth
\usepackage{xr}
\usepackage{algorithm}
\usepackage{algorithmic}

\floatstyle{plain}
\newfloat{Algorithm}{thp}{lop}
\floatname{Algorithm}{Algorithm}

\def\1{1\!{\rm l}}

\theoremstyle{definition}

\newtheorem{assumption}{Assumption}
\newtheorem{remark}{Remark}

\newtheorem{theorem}{Theorem}
\newtheorem{corollary}{Corollary}
\newtheorem{lemma}{Lemma}

\newcommand{\dt}{\text{d}}

\newcommand{\argmax}{\operatornamewithlimits{argmax\,}}

\newcommand{\bz}{\boldsymbol{z}}
\newcommand{\by}{\boldsymbol{y}}
\newcommand{\bw}{\boldsymbol{w}}
\newcommand{\bx}{\boldsymbol{x}}

\def \cut{\text{cut}}

\def \dt {\mathrm{d}}

\startlocaldefs
\numberwithin{equation}{section}
\theoremstyle{plain}

\endlocaldefs

\begin{document}

\begin{frontmatter}
\title{Cutting feedback and modularized analyses in generalized Bayesian  inference\support{This work was supported by the Australian Research Council and a Singapore	Ministry of Education Academic Research Fund Tier 1 grant.}}
\runtitle{Generalized Cut Posteriors}

\begin{aug}
\author{\fnms{David T.} \snm{Frazier}\thanksref{addr1}\ead[label=e1]{david.frazier@monash.edu}},
\and
\author{\fnms{David J} \snm{Nott}\thanksref{addr2}\ead[label=e2]{standj@nus.edu.sg}}

\runauthor{D.T. Frazier, and D.J. Nott}

\address[addr1]{Department of Econometrics and Business Statistics, Monash University, Australia.
    \printead{e1} % print email address of "e1"
}

\address[addr2]{Department of Statistics and Applied Probability, National University of Singapore, Singapore. 
	\printead{e2}
}

\end{aug}
\begin{abstract}
This work considers Bayesian inference under misspecification for complex  statistical models comprised of simpler submodels, referred to as modules, that are coupled together.  Such ``multi-modular" models often arise when combining information from different data sources, where there is a module for each data source.  When some of the modules are misspecified, the challenges of Bayesian inference under misspecification can sometimes be addressed by using ``cutting feedback" methods, which modify conventional Bayesian inference by limiting the influence of unreliable modules. Here we investigate cutting feedback methods in the context of generalized posterior distributions, which are built from arbitrary loss functions, and present novel findings on their behaviour.  We make three main contributions. First, we 
describe how cutting feedback methods can be defined in the generalized
Bayes setting, and discuss the appropriate scaling of the loss functions for different modules to each other and the prior.  Second, we 
derive a novel result about the large sample behaviour of the posterior for a given module's parameters conditional on the parameters of other modules.  
This formally justifies the use of conditional Laplace approximations, which provide better approximations of conditional posterior distributions compared to conditional distributions from a Laplace approximation of the joint posterior.  
Our final contribution leverages the large sample approximations 
of our second contribution to provide convenient diagnostics for
understanding the sensitivity of inference to the coupling of the modules, 
and to implement a new semi-modular posterior approach for conducting
robust Bayesian modular inference.  
The usefulness
of the methodology is illustrated in several benchmark examples from
the literature on cut model inference.
		
\end{abstract}

\end{frontmatter}

\section{Introduction}

Complex statistical models are sometimes composed of smaller sub-models, which we call modules, that are interconnected. This modular structure is common when integrating information from multiple data sources, where each data source is associated with a separate sub-model. When a model with a modular structure is correctly specified, Bayesian inference has some desirable
properties, regardless of the number or complexity of the modules. However, when there is misspecification, conventional Bayesian inference may need to be adapted to account for it.
This paper explores some new forms of a 
method called ``cutting feedback" for modified Bayesian inference
under misspecification.  

It is well-known that misspecification of an assumed model
compromises the use and interpretation of Bayesian inference; see, e.g., \cite{grunwald2012safe} for examples. 
Nonetheless, when dealing with a multi-modular model, 
a researcher may suspect that only some modules are grossly misspecified. In such cases, modified Bayesian analyses can be used to preserve valid inference for parameters in the correctly specified modules. This can make model criticism easier and ensure that estimates of parameters in the misspecified modules retain a useful interpretation \citep{liu+bb09}.  
These are some of the goals of the cutting feedback methods 
which are the focus of this paper, which attempt to limit the
influence of unreliable modules.  
To understand better the wide-ranging applications of cutting feedback 
and modularized Bayesian
inference, we recommend the papers by \cite{jacob2017better} and 
\cite{liu+bb09}, with the latter paper focusing on applications in the analysis of computer models.

The current literature on cutting feedback mainly focuses on fully specified parametric models. 
However, if a parametric model is misspecified, researchers can still produce useful Bayesian inferences by using a posterior based 
on a loss function that captures the features of the data that are most important. 
Such generalized Bayesian inference methods (see, for example, \cite{bissiri2016general}), have become increasingly popular in statistical inference.  They recover conventional Bayesian inference as a special case 
when the
loss function used in their construction is the negative log likelihood.  
This paper combines the use of cutting feedback methods with generalized Bayesian inference, resulting in  
an  attractive approach to Bayesian modular inference.  Our framework
allows a  
targeted loss function to be used for modules which are misspecified,
instead of relying on the negative log likelihood function. Meanwhile, we can continue to use the negative log likelihood function as the loss for modules that are well specified.  The generalized Bayes perspective on modular
inference is useful in model improvement.  Starting
with a flawed parametric model specification, 
we can replace 
the negative log likelihood for suspect modules with other loss functions
to see whether this 
resolves any incompatibility between the ``cut posterior" produced
by cutting feedback methods and full posterior inferences.

Our work makes three main contributions to the literature on generalized Bayesian inference and cutting feedback. 
Firstly, we describe how to define cutting feedback in the generalized
Bayesian setting, and discuss how to appropriately scale loss functions
for different modules to each other and the prior.  
Secondly, we derive a novel large sample result that allows us to express the posterior for the parameters of a 
given module conditional on the parameters of the remaining modules.  In contrast, the only existing result on 
the large sample behaviour of cut posteriors of which we are aware \citep{pompe2021asymptotics} presents 
a joint analysis of the cut posterior.  \cite{pompe2021asymptotics} also discuss a novel posterior bootstrap approach to
cut posterior computation.  
As we argue in Section \ref{sec:cutting}, 
a normal approximation to the joint cut posterior provides only limited insight into propagation of uncertainty
in cutting feedback, because conditioning on a subset of variables in a multivariate normal distribution
results in a conditional covariance matrix that doesn't depend on the values of the conditioning variables.
In contrast, our results justify normal approximations
for conditional posterior distributions where covariance matrices change with the values 
for the conditioning variables, giving useful insights into uncertainty propagation in cut posteriors.
Our new result is also applicable to general loss functions, and only requires weak smoothness conditions.

Finally, we use the large sample approximations provided 
in our second contribution to develop easily computable diagnostics 
for understanding the coupling of the modules, and to 
implement a new ``semi-modular" posterior for conducting 
robust modular inference.   
Semi-modular inference
\citep{carmona2020semi} partially cuts feedback, interpolating between
inferences based on the cut and full posterior according to a tuning parameter.
The challenges of cut posterior computation
also apply to semi-modular inference, with the key difficulty being the 
evaluation of 
an intractable marginal likelihood term.  Estimation of the semi-modular (and cut) posterior is 
often done using a computationally burdensome nested Markov chain Monte Carlo (MCMC) method,
and our novel semi-modular posterior can be computed efficiently using
the large sample approximations we develop, delivering similar results to the semi-modular posterior of \cite{carmona2020semi}. See Section \ref{sec:semimod} for further details. 
We illustrate the above diagnostics and semi-modular posterior in two benchmark examples found in the literature on cutting feedback. 

\noindent\textbf{Notation.} Here we define notation used in the remainder of the paper. The term $\|\cdot\|$ denotes the Euclidean norm, while $|\cdot|$ denotes the absolute value function. $C$ denotes an arbitrary positive constant that can change from line-to-line. 
For $x=(x_1^\top,x_2^\top)^\top\in\mathbb{R}^d$ and a function $f:\mathbb{R}^d\mapsto \mathbb{R}$, we let $\nabla_x f(x)$ denote the gradient of $f(x)$ wrt $x$, and $\nabla_{xx}^2f(x)$ the Hessian. Let $N\{\mu,\Sigma\}$ denote the normal distribution with mean $\mu$ and covariance matrix $\Sigma$, with $N\{x;\mu,\Sigma\}$ the corresponding 
normal density at the point $x$. For $\mathcal{D}$ some known distribution, and $x=(x_1^\top,x_2^\top)^\top\in\mathbb{R}^d$ a $d$-dimensional random variable, the notation $x\sim \mathcal{D}$ signifies that the law of $x$ is $\mathcal{D}$, while $x_1|x_2\sim \mathcal{D}$ signifies that the conditional law of $x_1$ given $x_2$ is $\mathcal{D}$. The measure $P^{(n)}_{0}$ denotes the true unknown probability measure generating the data, and $\Rightarrow$ denotes weak convergence (under $P^{(n)}_0$). { } %For two sequences $a_n,b_n$, we say that  $a_n\lesssim b_n$ if for some constant $C>0$ $a_n\le Cb_n$ for all $n$ large enough. The term $a_n\asymp b_n$ means that $a_n\lesssim b_n$ and $b_n\lesssim a_n$. 

\section{Motivation and Framework}

Modifying Bayesian inference to limit the influence of a suspect module is 
the main idea of cutting feedback methods.   
But what is a module exactly, and how is cutting feedback defined for multi-modular models of arbitrary complexity?  This is not a settled question in the current literature.  
Recent work by \cite{liu+g22} has provided a 
first step towards clarity, where the authors define modules 
based on the representation of a Bayesian model in terms of a 
directed acyclic graph (DAG) and a partitioning of the observable 
quantities.  It is fair to say, however, that different general formulations
of modular inference are still being explored.    

In previous work, the most general approach to
cutting feedback methods has involved an
``implicit" definition through modification of an MCMC
algorithm designed to sample the conventional posterior distribution.  
One implementation of this approach is through the \texttt{cut} function
of the WinBUGS and OpenBUGS software packages 
\citep{lunn+bsgn09}.  If a Bayesian model is defined through
a DAG, and a Gibbs sampler is considered for sampling the 
posterior distribution using the DAG parameter nodes as blocks, 
then ``cuts" can be defined for some links of the graph.  Each cut
corresponds to leaving out a certain term in the joint model when 
forming the full conditional posterior density for one of the parameter nodes.   
Once modified full conditional distributions
have been constructed, a modified Gibbs sampler iteratively samples
from these, and the cut posterior distribution is defined as the stationary
distribution of the resulting Markov chain.  See \citet{lunn+bsgn09} 
or \cite{plummer2015cuts} for a more detailed description.  

\citet{lunn+bsgn09} note that the modified full conditional distributions
are not the full conditional distributions of any well-defined joint
distribution but argue that the use of such inconsistent conditional
distributions can be sensible.  If modified Gibbs steps are replaced
by Metropolis-within-Gibbs updates in the sampling process, 
\cite{plummer2015cuts} observed that the stationary distribution
of the Markov chain can depend on the proposal used, and went on to 
define a ``two-module" system where an explicit definition of the
cut posterior distribution can be given, clarifying some aspects of the
implicit cut approach.   This two module system is general enough for
many applications of Bayesian modular inference in which there might be one suspect model component of particular concern.  This two module system is also fundamental to
the recent work of \cite{liu+g22} where multi-modular systems
and cut posteriors are defined generally.  \cite{liu+g22} define two module systems first,
based on a partitioning of the observables into two parts, and then 
consider recursively splitting existing modules into two in order to define 
more complex multi-modular representations.  
In what follows, we will focus our discussion on cutting feedback in two-module systems, given their usefulness in applications and their role in defining multi-modular models with more than two modules. We define modules and
cutting feedback precisely in the context of this two module system, and
refer the interested reader to 	\cite{liu+g22}
for a more general discussion.

\subsection{Two module system} \label{sec:twomod}

The ``two module" system of \cite{plummer2015cuts} considers two
data sources, denoted here as $\bz$ and $\bw$.  
The data $\bz$ consists of $n_1$ observations $\bz=(z_{1},\dots,z_{n_1})^\top$, $z_i\in\mathcal{Z}$, and $\bw$ consists of 
$n_2$ observations $\bw=(w_1,\dots,w_{n_2})$, $w_i\in\mathcal{W}$, 
and we write $n=n_1+n_2$. Let $\by=(\bz^\top,\bw^\top)^\top$ denote the entire set of observed data. 
A potentially misspecified statistical model for $\by$ is postulated that depends on parameters $\theta\in\Theta\subseteq\mathbb{R}^{d_\theta}$, where $\theta=(\varphi^\top,\eta^\top)^\top$, with $\varphi\in\Phi\subseteq\mathbb{R}^{d_\varphi}$, $\eta\in\mathcal{E}\subseteq\mathbb{R}^{d_\eta}$, and $d_\theta=d_\varphi+d_\eta$. 
Prior beliefs over $\Theta$ are represented by the prior density $\pi(\theta)=\pi(\varphi)\pi(\eta|\varphi)$. 

\cite{plummer2015cuts} considers Bayesian inference on $\theta$ in cases where the distribution of $\bz$ depends on $\varphi$, with density $p(\bz|\varphi)$, while the distribution of $\bw$ depends on $\eta$ and $\varphi$, with density
$p(\bw|\eta,\varphi)$.  In other words, the models for the two data sources have a shared dependence on $\varphi$ (a global parameter) whereas $\eta$ appears only in the model for $\bw$.  We define ``modules" here as subsets
of terms in the joint Bayesian model.   Here there are two modules, 
with the first consisting of the likelihood term $p(\bz|\varphi)$ 
and prior $\pi(\varphi)$, and the second consisting of 
the likelihood $p(\bw|\eta,\varphi)$ and conditional prior 
$\pi(\eta|\varphi)$.  The structure of the model is modular in the sense
that valid Bayesian inference about $\varphi$ can be obtained based on 
module one only (i.e. we can use $\pi(\varphi|\bz)\propto \pi(\varphi)p(\bz|\varphi)$), while given a value of $\varphi$, valid conditional Bayesian
inference for $\eta|\varphi$ can be obtained based on module two only
(i.e. we an use $\pi(\eta|\varphi,\bw)\propto \pi(\eta|\varphi)p(\bw|\eta,\varphi)$).  
The graphical structure of the model is given in Figure \ref{two-module}, with the nodes to the left of the red dashed line comprising module one, and the nodes to the right comprising module two, where it is assumed in the figure that
$\pi(\eta|\varphi)$ does not depend on $\varphi$ for simplicity.  \vspace{-.3cm}

\begin{figure}[h]
	\centerline{\includegraphics[width=40mm]{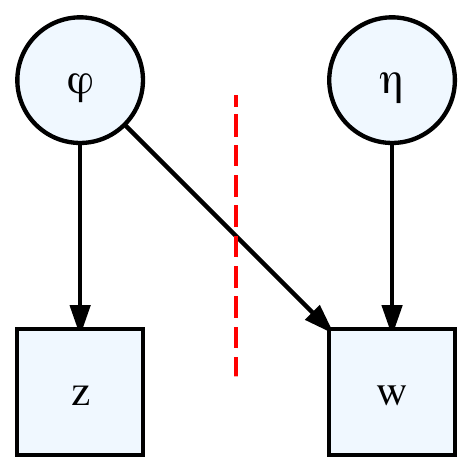}}
	
	\caption{\label{two-module}Graphical structure of the two-module system.  The red dashed line indicates the cut.}
\end{figure}
\vspace{-.5cm}

We assume that there is high confidence in the accuracy of the model for $\bz|\varphi$ in the first
module, but it is uncertain that the model for $\bw|\eta,\varphi$ in the second model is adequate.
Consequently, if we were to conduct standard Bayesian inference on $\theta$ using both modules, our inferences on the shared parameter $\varphi$ could be contaminated by misspecification of the second module, and any useful interpretation
for our inferences about $\eta$ may also be compromised if the parameters $\varphi$ do not have their
intended meaning. See Section \ref{sec:examples} for examples.

We will discuss two methods that can guard against compromised inferences on $\varphi$ due to potential misspecification of the second module. The first method involves using a loss function rather than a parametric model 
to capture the important features of the data for the second module;  a generalized posterior is constructed based on the loss function for the parameters of interest. The second approach is to employ cutting feedback methods. Generalized Bayesian methods and cutting feedback are not used here as approximations to conventional Bayesian inference;  they are 
alternative inferential approaches 
intended to address the issue of misspecification and
having a sound statistical justification in their own right.  
Approximate methods for computation may be of interest, but this is 
discussed later in Section 3, based on the asymptotic results we develop there.
In this article, we aim to combine cutting feedback and generalized Bayesian updating to produce robust Bayesian inferences on $\theta$, and we
explain these concepts next. 

\subsection{Generalized Posteriors}
When the model is misspecified, standard Bayesian approaches can deliver  inferences that are poor or unreliable (see, e.g., \cite{grunwald2017inconsistency} for specific examples, as well as \cite{kleijn2012bernstein} for general results in parametric models). 
Specifying full probabilistic models for complex data can be difficult,
and it would be attractive if Bayesian inference could be done only
for parameters of interest appearing in a loss function.
Under some mild conditions on the loss, 
\cite{bissiri2016general} justify a Bayesian analysis 
in this setting 
in which the likelihood in the usual Bayesian update is replaced
with a ``loss likelihood" with a highly constrained form.  
The target parameter of interest is the population minimizer of the
loss.  

In a standard generalized posterior analysis without modular structure, 
there is a parameter $\theta$ and data $\bx=(x_1,\dots, x_n)^\top$ say.  
The prior $\pi(\theta)$ is to be updated into a generalized posterior
$\pi(\theta|\bx)$, where the belief update depends on $\bx$ only
through a loss function $q_n(\theta)=\sum_{i=1}^n q(x_i;\theta)$, where
$q(x_i;\theta)$ is the loss for the $i$th observation.  
A remarkable argument in \cite{bissiri2016general} specifies
the form that the belief update must take, under some mild conditions.  
They consider the requirement of order coherence, where
if the data $\bx$ are split into two parts and an update is done
sequentially, then the result should be the same as if a single update
were done using all the data.  Order coherence is enough to determine
the form of the generalized posterior density, which is 
$$\pi(\theta|\bx)\propto \pi(\theta)\exp\left\{-\nu q_n(\theta) \right\},$$
where $\nu\ge0$ is called the learning rate, and scales the information
in the loss function appropriately relative to the information in the prior.  While the generalized Bayesian update of \cite{bissiri2016general}
is motivated by Bayesian 
notions of coherence, the choice of learning rate gives
the opportunity to bring in other considerations such as information
matching in the update \citep{holmes+w17,lyddon+hw19} or achieving good frequentist
performance for estimating functionals of interest \citep{syring+m18}.   
A generalization of the arguments in \cite{bissiri2016general} relevant
to the justification of parametric cutting feedback methods is discussed
in \cite{nicholls2022valid}.  Generalized Bayesian updating is also related to PAC-Bayes methods;  
see \cite{alquier21} for an introduction.

Consider now the case of modular Bayesian inference in the two
module system.  The decomposition of the statistical model into two distinct modules, containing data $\bz$ and $\bw$ respectively, implies that we are free to choose separate loss functions for each module. Let $\ell:\mathcal{Z}\times\Phi\rightarrow\mathbb{R}$ denote the loss function for module one, 
involving a parameter $\varphi$, and $m:\mathcal{W}\times\mathcal{E}\times\Phi\rightarrow\mathbb{R}$ denote the loss function for module two, involving parameters $\eta$ and $\varphi$. 
In the following we write
\begin{equation}
	Q_n(\theta)=L_{n_1}(\varphi)+M_{n_2}(\eta,\varphi),\; L_{n_1}(\varphi)=-\sum_{i=1}^{n_1}\ell(z_i,\varphi),\; M_{n_2}(\eta,\varphi)=-\sum_{i=1}^{n_2}m(w_i,\eta,\varphi), \label{eq:loss}
\end{equation} 
so that $-L_{n_1}(\varphi)$ and $-M_{n_2}(\eta,\varphi)$ are the empirical loss
	functions for the first and second modules respectively. When the two sample sizes are equal, i.e., $n_1=n_2$, we abuse notation and simply denote the criteria as $L_n(\varphi)$ and $M_n(\eta,\varphi)$. 

Consider first a belief update of the prior density $\pi(\varphi)$ using 
$\bz$ and the first module loss function $\ell(\cdot)$.  
The order coherence argument of \cite{bissiri2016general} implies
that the generalized posterior density $\pi(\varphi|\bz)$ must take
the form
$ \pi(\varphi|\bz)\propto \pi(\varphi)\exp\left\{\nu L_{n_1}(\varphi)\right\},$ 
where $\nu\ge0$ is a learning rate for the first module that needs to be chosen.
If the loss function is the negative log-likelihood, and we 
take $\nu=1$, this is the conventional Bayesian update.  

Once $\pi(\varphi|\bz)$ is obtained, suppose we now take $\pi(\theta|\bz)=\pi(\varphi|\bz)\pi(\eta|\varphi)$ as the ``prior" for a Bayesian update 
using the information in the second module.  
Again following the 
order coherence argument of \cite{bissiri2016general}, and its extensions in \cite{nicholls2022valid}, the
generalized posterior density $\pi(\theta|\bz,\bw)$ given $\bz$ and $\bw$
must take the form
\begin{align} 
	\pi(\theta|\bz,\bw) & \propto \pi(\theta|\bz) \exp\{\nu' M_{n_2}(\eta,\varphi)\} \nonumber \\
	& \propto \pi(\varphi)\pi(\eta|\varphi)\exp\{\nu L_{n_1}(\varphi)+\nu' M_{n_2}(\eta,\varphi)\}, \label{fullposterior}
\end{align}
where $\nu'\ge0$ is an additional learning rate.  As before, 
if $\nu'=1$ and the loss function $m(\cdot)$ is the log-likelihood, 
this is a conventional Bayesian update using the data for the
second module.  
For the full belief update \eqref{fullposterior}, it takes the form of
Bayesian updating where the likelihood has been replaced by the 
loss likelihood $\exp\{\nu L_{n_1}(\varphi)+\nu' M_{n_2}(\eta,\varphi)\}$.  
If the two loss terms $L_{n_1}(\varphi)$ and $M_{n_2}(\varphi,\eta)$
are of the same type and hence on the same scale, 
then it could make sense to choose $\nu=\nu'$, and
we would obtain a loss likelihood with a single learning rate $\nu$ and 
of the customary form in a generalized
Bayesian analysis, 
$\exp\left\{\nu Q_n(\theta)\right\}$, 
where $-Q_n(\theta)$ is the overall empirical loss.  

In generalized Bayesian inference the choice of the learning rate is
very important, and this is true in the case of modular inference 
considered here also.
See \cite{wu2020comparison} for a review and comparison
of different methods.    
Generalizing similar ideas to \cite{holmes+w17} and \cite{lyddon+hw19}, later we suggest choosing $\nu$ and $\nu'$ based on an information matching
argument.  It is often 
not necessary in the applications we consider to estimate the first module learning rate $\nu$:  
if the first module is specified
through a probabilistic model and we are confident in this specification, 
$\nu=1$ is the natural choice.  
However, we discuss the choice of learning rates
in Section 3.2 in a general way, addressing the situation where it may be desired
to choose both $\nu$ and $\nu'$.      

Our later theoretical results will be written using a loss likelihood with a single learning rate where $\nu=\nu'$.  There is no loss of generality
in this, or even in omitting learning rates altogether in the theoretical discussion, 
since any learning rates can be absorbed into the definition of the
loss function.

%Inference in many multivariate statistical models can be based on criterion functions of the form \eqref{eq:loss}. 
%For instance, in models for which composite likelihood methods are convenient \citep{varin2011overview} the
%composite likelihood could be used to define an appropriate criterion function.

\subsection{Cutting Feedback with Generalized Posteriors}\label{sec:cutting}

Our confidence in the accuracy of the first module means that the criterion $L_{n_1}(\varphi)$ can be chosen as the log-likelihood.  However, since we are working with generalized posteriors, we only maintain that $L_{n_1}(\varphi)$   produces ``reliable inferences'' for $\varphi$. Our lack of confidence in the specification of the second module means we are concerned that incorporating this module may contaminate our inferences for $\varphi$. In such situations, cutting feedback methods (see, e.g., \citealp{plummer2015cuts}) can be used to mitigate the impact of misspecification. 

In the two module system discussed in Section \ref{sec:twomod} for a probabilistic
model, the first module consists of the terms $\pi(\varphi)$ and $p(\bz|\varphi)$ in the joint Bayesian model, and the second module consists of 
$p(\eta|\varphi)$ and $p(\bw|\eta,\varphi)$.  In a generalized Bayesian
analysis, module one consists of $p(\varphi)$ and the loss likelihood term
$\exp\{\nu L_{n_1}(\varphi)\}$, and module two consists of 
$\pi(\eta|\varphi)$ and the loss likelihood $\exp\{\nu M_{n_2}(\eta,\varphi)\}$, 
if a single learning rate is assumed for both modules.

Generalized Bayesian analyses have been used in the context
of two module system previously, but only as a justification for parametric
cutting feedback methods when a probabilistic model is specified.
\cite{carmona2020semi} considered order coherence
for cut and semi-modular inference methods, and    
\cite{nicholls2022valid} observed
that the implicit loss function used in these approaches 
is not additive as required in the theory of \cite{bissiri2016general}.  \cite{nicholls2022valid} 
generalize the existing theory to ``prequentially
additive" loss functions, which is enough to justify standard parametric
cut inference as valid and order coherent generalized Bayesian updating.  
In contrast to this work, our aim is not to justify cutting feedback methods for
probabilistic multi-modular models as coherent in some sense, but to consider
situations where there may be no probabilistic model for the data, but only
loss functions to connect module data to parameters. 

To present cutting feedback for generalized posteriors, decompose $\pi(\theta|\by)$ in \eqref{fullposterior} as the product of a marginal posterior for $\varphi|\bz$, a conditional posterior for $\eta|\bw,\varphi$, and a ``feedback term'': 
\begin{equation}\label{eq:decomp}
	\pi(\theta|\by)=\pi_{\cut}(\varphi|\bz)\pi(\eta|\bw,\varphi)\tilde{p}(\bw|\varphi),
\end{equation}
where $\pi_{\cut}(\varphi|\bz)\propto \pi(\varphi)\exp\{\nu L_{n_1}(\varphi)\}$, $\pi(\eta|\bw,\varphi):= {\pi(\eta|\varphi)\exp\{\nu M_{n_2}(\eta,\varphi)\}}/{m_\eta(\bw|\varphi)},$ and
\begin{equation}\label{eq:post}
	\begin{aligned}
		\tilde{p}(\bw|\varphi)\propto m_\eta(\bw|\varphi),\;\;
		%\pi(\varphi|\bz)&= \frac{\pi(\varphi)\exp\left\{\nu L_{n}(\varphi)\right\}}{\int_{\Phi}\pi(\varphi)\exp\{\nu L_n(\varphi)\}\dt\varphi},\\
		m_\eta(\bw|\varphi)=\int_{\mathcal{E}}\pi(\eta|\varphi)\exp\{\nu M_{n_2}(\eta,\varphi)\}\dt\eta
	\end{aligned}.
\end{equation}
The feedback term $\tilde{p}(\bw|\varphi)$ derives its name 
from representing the influence of module two on the marginal posterior for
$\varphi$.  To understand this better, consider integrating 
out $\eta$ in \eqref{eq:decomp}, to obtain
$\pi(\varphi|\by)=\pi_{\text{cut}}(\varphi|\bz)\widetilde{p}(\bw|\varphi)$.  
Since $\pi_{\text{cut}}(\varphi|\bz)$ represents the posterior density
for $\varphi$ based only on the first module data $\bz$, we see that
$\widetilde{p}(\bw|\varphi)$ modifies this posterior based on the second 
module data to give the $\varphi$ marginal of $\pi(\theta|\by)$.  
Dropping the feedback term $\tilde{p}(\bw|\varphi)$ in 
\eqref{eq:decomp} produces a ``generalized cut posterior'':
$$\pi_{\text{cut}}(\theta|\bz,\bw):= {\pi}_{\cut}(\varphi|\bz)\pi_{}(\eta|\bw,\varphi). $$
In this joint cut posterior, marginal posterior inferences for $\varphi$ are obtained
based on module one only, and the conditional posterior density of $\eta$
given $\varphi$ is the same as for $\pi(\theta|\by)$ and based on module two
only.  Our discussion of
cut inference is in the generalized Bayesian framework, but if we
use negative log likelihood as the loss for an assumed probabilistic model, 
our definition of the cut posterior reduces to the conventional one in the
literature.

Obtaining samples from the cut posterior $\pi_{\text{cut}}(\theta|\bz,\bw)$ can be challenging. Since $$
\pi_{\text{cut}}(\theta|\bz,\bw)\propto \pi(\varphi)\exp\{\nu L_{n_1}(\varphi)\}\frac{\pi(\eta|\varphi)\exp\{\nu M_{n_2}(\eta,\varphi)\}}{m_\eta(\bw|\varphi)},
$$ if MCMC is used to sample from $\pi_{\text{cut}}(\theta|\bz,\bw)$, we must evaluate the term ${m_\eta(\bw|\varphi)}$.  This term is similar to a``marginal likelihood'' for $\eta$ conditional on a fixed $\varphi$, and is generally not available in closed form outside of toy examples. In principle, even though we are in the case of generalized posteriors, the  computationally intensive methods proposed by \cite{plummer2015cuts}, and \cite{jacob2017better} to deal with the intractable term ${m_\eta(\bw|\varphi)}$ could be used to sample from the cut posterior. 

%Since we are confident  that $L_n(\varphi)$   produces ``reliable inferences'' for $\varphi$, an alternative to using the  generalized (marginal) posterior $\pi_{}(\varphi|\bz,\bw)$  is to base our inferences for $\varphi$ on the cut posterior  $\pi_\cut(\varphi|\bz)$. In contrast to $\pi_\cut(\varphi|\bz)$, the potential model misspecification of the second module , $\bw|\eta,\varphi)$, implies that inferences based on $\pi_{}(\varphi|\bz,\bw)$ may be inferior to those based on  $\pi_\cut(\varphi|\bz)$. In such case, cutting feedback has been suggested as a possible avenue for conducting Bayesian inference. 

While sampling from $\pi_{\text{cut}}(\theta|\bz,\bw)$ is difficult, draws from ${\pi}_{}(\eta|\bw,\varphi)$ for any $\varphi$ can be made without the need 
to compute $m_\eta(\bw|\varphi)$. This suggests the following 
sequential algorithm to obtain draws from $\pi_{\text{cut}}(\theta|\bw,\bz)$: first, sample $\varphi'\sim \pi_{\cut}(\varphi|\bz)$; then, sample $\eta'\sim{\pi}_{}(\eta|\bw,\varphi')$.  
At the first stage, draws from $\pi_{\cut}(\varphi|\bz)$ could be obtained
by running an MCMC chain targeting the posterior density 
$\pi_{\cut}(\varphi|\bz)$.  The conditional draws of $\eta$ given $\varphi$
are then performed by running a separate MCMC chain for each sample, 
which is computationally burdensome.    The approach is reminiscent
of multiple imputation algorithms, and was originally suggested by \cite{plummer2015cuts}, who also discussed
a related tempering method of similar computational complexity.  
The sequential sampling approach above can also be thought of as 
implementing a modified Gibbs sampling algorithm
with blocks $\varphi$ and $\eta$, but where the likelihood term from the second module is dropped when
forming the full conditional distribution for $\varphi$.  
As mentioned earlier, the resulting modified conditional distributions
are not the full conditional distributions of any joint distribution in general, and if we attempt to replace
the usually intractable direct sampling of the modified conditional distributions with Metropolis-within-Gibbs steps, then
the stationary distribution of the MCMC sampler depends on the proposal used.  
A number of other authors
have investigated computation for cutting feedback \citep{jacob+oa17,liu+g20,yu2021variational,carmona+n22} and this remains
an active area of research.

The sequential definition
of the cut posterior distribution in the two-module system suggests that the statistical analysis of cut procedures
should study the marginal cut posterior density $\pi_\cut(\varphi|\bz)$ to understand cut inferences for
$\varphi$, and the conditional posterior of $\pi(\eta|\bw,\varphi)$ to understand how uncertainty about 
$\varphi$ propagates to marginal cut inferences about $\eta$.  This is the strategy we follow in the next 
section.  Such an analysis is complicated by the fact that $\pi_{\text{cut}}(\theta|\bz,\bw)$ does not arise as a posterior for a generative model, 
and therefore we must use techniques employed in the study of generalized 
posteriors to analyze $\pi_{\text{cut}}(\theta|\bz,\bw)$.

\section{The Behavior of $\pi_{\text{cut}}(\theta|\bz,\bw)$}
In this section, we explore the behavior of $\pi_{\text{cut}}(\theta|\bz,\bw)$ by separately analysing $\pi_\cut(\varphi|\bz)$, and then analysing $\pi_{}(\eta|\bw,\varphi)$, when we condition on an observed value of $\varphi$ within the high probability region of $\pi_{\cut}(\varphi|\bz)$. This yields useful insights into the behavior of cut posteriors and allows us to develop new diagnostic tools for examining these posteriors.  The normal approximations implied by our asymptotic results 
are also valuable for cut posterior computation.  As discussed in Section 2.3, 
a common way to sample the cut joint posterior distribution involves a nested
MCMC scheme where a separate MCMC chain is run to draw a
sample of $\eta$ from its posterior conditional density 
for each marginal cut posterior sample $\varphi$.  If this MCMC step
can be replaced by a draw from a normal approximation, or the normal
approximation is used to obtain a good proposal density for MCMC or
importance sampling, then this can reduce the computational burden
of commonly used methods for cut posterior computation.

%We then use this result to produce methods that allow us to easily explore the behavior of $\pi_{}(\eta|\bw,\varphi)$, and gauge the possible impact of incorporating information from the second module ($\bw|\eta,\varphi)$ into the cut posterior for  $\varphi$. 

\subsection{Maintained Assumptions and Main Results}\label{app:assum}

The assumptions used to obtain the following theoretical result constitutes a  generalization of the assumptions often employed to analyze the behavior of generalized posteriors; see \cite{miller2021asymptotic} for an in-depth discussion. 
We consider an asymptotic
regime in which there is a limiting ratio for the sample sizes
for the two modules, $\zeta:=\lim_{n\rightarrow\infty} n_1/n_2$, $0<\zeta<\infty$.
First, we consider the cut posterior $\pi_{\cut}(\varphi|\bz)$, and maintain the following conditions, which are sufficient to demonstrate posterior concentration.
\begin{assumption}\label{ass:ident_L}(i) There exist $\mathbb{L}(\varphi)$ such that $\sup_{\varphi\in\Phi}|n_1^{-1}L_{n_1}(\varphi)-\mathbb{L}(\varphi)|=o_p(1).$ 
	(ii) There is a unique $\varphi^\star\in\text{Int}(\Phi)$ such that for every $\delta>0$ there exists $\epsilon(\delta)>0$ so that $\sup_{\|\varphi-\varphi^\star\|\ge\delta}\{\mathbb{L}(\varphi)-\mathbb{L}(\varphi^\star)\}\le-\epsilon(\delta).$	(iii) $\pi(\varphi)$ is continuous on $\Phi$, with $\pi(\varphi^\star)>0$, and $\int_{\Phi}\|\varphi\|\pi(\varphi)\dt\varphi<\infty$. (iv) For an arbitrary $\delta>0$, and $\|\varphi-\varphi^\star\|\le\delta$, $\mathbb{L}(\varphi)$ and $L_{n_1}(\varphi)$ are twice continuously differentiable, with $\sup_{\|\varphi-\varphi^\star\|\le\delta}\|\nabla_{\varphi\varphi}^2L_{n_1}(\varphi)/n_1-\nabla_{\varphi\varphi}^2\mathbb{L}(\varphi)\|=o_p(1)$, and $-\nabla_{\varphi\varphi}^2\mathbb{L}(\varphi^\star)$ positive-definite. (v)  $\nabla_\varphi L_{n_1}(\varphi^\star)/\sqrt{n_1}=O_p(1)$.
\end{assumption} 
\begin{remark}
	Assumption \ref{ass:ident_L} is similar to the standard conditions employed to obtain posterior asymptotic normality, see, e.g., \cite{lehmann2006theory} (Ch 6.8.1) or Theorem 4 in \cite{miller2021asymptotic}, but allows $L_{n_1}(\varphi)$ to be an arbitrary criterion function. Assumptions (i)-(iii) allow for posterior concentration onto $\varphi^\star$, while the smoothness conditions in (iv)-(v) ensure this concentration occurs in a Gaussian manner. Assumption \ref{ass:ident_L} (iv) and (v) are maintained for simplicity, and can be replaced with `stochastic differentiability' assumptions at the introduction of additional technicalities. 
\end{remark}

Define $\Sigma_{11}:=\nabla_{\varphi\varphi}\mathbb{L}(\varphi^\star)$, $Z_{n_1}(\varphi^\star):=-\Sigma_{11}^{-1}\nabla_\varphi L_{n_1}(\varphi^\star)/\sqrt{n_1},$ the local parameter $\phi:=\sqrt{n_1}(\varphi-\varphi_\star)$ and its posterior $
{\pi}_{}(\phi|\bz)=\pi_{}(\varphi^\star+\phi/\sqrt{n_1}|\bz)/\sqrt{n_1}^{d_\varphi},$ which has support $\Phi_{n_1}:=\{\phi:\sqrt{n_1}(\varphi-\varphi^\star)\in\Phi\}$. Lemma \ref{lem:varphi}  states that the cut posterior $\pi_\cut(\phi|\bz)$ behaves like a Gaussian density with mean $Z_{n_1}(\varphi^\star)$, and covariance $\left[\nu\Sigma_{11}\right]^{-1}$. %An equivalent interpretation of Lemma \ref{lem:varphi} is that the cut posterior for $\varphi$ concentrates mass in a neighbourhood of $\varphi^\star$, and that this concentration occurs in a Gaussian manner. 

\begin{lemma}\label{lem:varphi}
	Under Assumption \ref{ass:ident_L},
	$
	\int_{\Phi_{n_1}}\|\phi\|\left|\pi_{\cut}(\phi|\bz)-N\{\phi;Z_{n_1}(\varphi^\star),\left[\nu\Sigma_{11}\right]^{-1}\}\right|\dt \phi=o_p(1).
	$
\end{lemma}

The conditioning value of $\varphi$ in $\pi_{}(\eta|\bw,\varphi)$ plays a key role in the behavior of posterior. To demonstrate this, we view $M_{n_2}(\eta,\varphi)$ as being indexed by a fixed $\varphi\in\Phi_\delta$, and to reinforce this perspective we 
use the notation $M_{n_2}(\eta|\varphi):=M_{n_2}(\eta,\varphi)$. 
Let $\Phi_\delta:=\{\varphi\in\Phi:\|\varphi-\varphi^\star\|\le\delta\}$ denote an arbitrary $\delta$-neighborhood of $\varphi^\star$, and consider the following regularity conditions on $M_{n_2}(\eta|\varphi)$. 

\begin{assumption}\label{ass:ident_M}(i) There exist $\mathbb{M}(\eta|\varphi)$ such that, for some $\delta>0$, $\sup_{\varphi\in\Phi_\delta,\eta\in\mathcal{E}}|{n_2}^{-1} M_{n_2}(\eta|\varphi)-\mathbb{M}(\eta|\varphi)|=o_p(1).$ 
	(ii) Given $\delta_1>0$, for each $\varphi\in\Phi_{\delta_1}$ there is an $\eta^\star_\varphi\in\text{Int}(\mathcal{E})$ such that for any $\delta_2>0$ there exist $\epsilon(\delta_1,\delta_2)>0$, so that $\sup_{\varphi\in\Phi_{\delta_1}}\sup_{\|\eta-\eta_\varphi^\star\|\ge\delta_2}\{\mathbb{M}(\eta|\varphi)-\mathbb{M}(\eta_\varphi^\star|\varphi)\}\le-\epsilon(\delta_1,\delta_2).$	
\end{assumption}

\begin{assumption}\label{ass:prior} (i)  For some $\delta>0$, and each $\varphi\in\Phi_\delta$, $\pi(\eta|\varphi)$ is continuous in $\eta$. (ii) $\sup_{\varphi\in\Phi_{\delta_1}}\int_{\mathcal{E}}\|\eta\|\pi(\eta|\varphi)\dt\eta<\infty$. 
	
\end{assumption}
\begin{remark}
	Assumption \ref{ass:ident_M} constitutes a set of conditions on the conditional loss function $M_{n_2}(\eta | \varphi)$, which, together with the prior condition in Assumption \ref{ass:prior}, ensure that for a fixed $\varphi\in\Phi_\delta$ the conditional posterior concentrates onto $\eta^\star_\varphi$.
	 These conditions imply that if we study $\pi(\eta\mid\varphi,\bw)$ in a neighbourhood of $\varphi^\star$, the conditional posterior concentrates mass near $\eta^\star_\varphi$. The form of this posterior means that this conditional interpretation of concentration is a more natural way of representing posterior behavior than the conventional joint analysis. 
	
	Technically, Assumption \ref{ass:ident_M}(i) assumes uniform convergence of the loss function, but where we restrict the $\varphi$ parameter to the neighbourhood $\Phi_\delta$. Of course, a sufficient condition for this would simply be uniform convergence of $M_{n_2}(\eta | \varphi)$ across $\mathcal{E}\times\Phi$. Assumption \ref{ass:ident_M}(ii) assumes that for each $\varphi\in\Phi_\delta$, the limit criterion $\mathbb{M}(\eta | \varphi)$ has a unique optimum. This ``conditioning on'' $\varphi\in\Phi_\delta$ then allows us to view $M_{n_2}(\eta | \varphi)$ as being indexed by a fixed parameter value, and we can then employ similar regularity conditions to those used in the study of frequentist estimation theory but at an arbitrary $\varphi\in\Phi_\delta$; we refer to \cite{Portier2016} for a discussion and several examples. {We remark that, in cases where $\mathbb{M}(\eta | \varphi)$ is known, the implicit function theorem could be used to obtain the mapping $\varphi\mapsto\eta_\varphi^\star$ either numerically or, if available, analytically, and would allow for immediate verification of Assumption \ref{ass:ident_M}(ii).} 

	Assumption \ref{ass:prior}(i) is a standard regularity condition, while Assumption \ref{ass:prior}(ii) implies that the conditional prior has sufficient moments. A sufficient condition for the latter condition is that the posterior $\pi(\eta | \varphi)=\pi(\eta)$ and $\int \|\eta\|\pi(\eta)\dt\eta<\infty$. 	
\end{remark}

\begin{assumption}\label{ass:expand}For some $\delta_1,\delta_2>0$, the following are satisfied. There exist a vector function $\Delta_{n_2}(\eta|\varphi)$, and matrix function $J(\eta|\varphi)$ such that
	$$
	M_{n_2}(\eta|\varphi)-M_{n_2}(\eta^\star_\varphi|\varphi)=(\eta-\eta^\star_\varphi)^\top\Delta_{n_2}(\eta^\star_\varphi|\varphi)-\frac{n_2}{2}(\eta-\eta^\star_\varphi)^\top J(\eta|\varphi)(\eta-\eta^\star_\varphi)+R_{n_2}(\eta,\varphi).
	$$ %The following conditions on $\Delta_{n_2}(\eta|\varphi)$, $J(\eta|\varphi)$ and $R_{n_2}(\eta,\varphi)$ are satisfied. 
	
	\noindent (i) for all $\varphi\in\Phi_{\delta_1}$, $\Delta_{n_2}(\eta^\star_\varphi|\varphi)/\sqrt{n_2}=O_p(1)$; 
	
	\noindent (ii) the map $\eta\mapsto J(\eta|\varphi)$ is continuous for all $\|\eta-\eta^\star_\varphi\|\le\delta_2$,  for each $\varphi\in\Phi_{\delta_1}$, and $J(\eta^\star_\varphi|\varphi)$ is positive-definite for each $\varphi\in\Phi_{\delta_1}$; 
	
	\noindent (iii) for any $\delta_2>0$,  $\sup_{\varphi\in\Phi_{\delta_1}}\sup_{\|\eta-\eta^\star_\varphi\|\le{\delta_2}}{R_{n_2}(\eta,\varphi)}/[{1+n_2\|\eta-\eta^\star_\varphi\|^2}]=o_p(1).$
\end{assumption}

\begin{remark}
	Assumption \ref{ass:expand} ensures that $M_{n_2}(\eta | \varphi)$ admits a valid quadratic expansion around $\eta_\varphi^\star$ for each $\varphi\in\Phi_\delta$; a sufficient condition for this is that, for each $\varphi\in\Phi$, $M_{n_2}(\eta | \varphi)$ is twice continuously differentiable in $\eta$, and that the matrix of second derivatives $-\nabla_{\eta\eta}M_{n_2}(\eta | \varphi)/{n_2}$ uniformly converges to its expected counterpart, denoted as $J(\eta | \varphi)$. Assumption \ref{ass:expand}(i) requires that the first term in the quadratic expansion is asymptotically bounded for each pair $(\eta_\varphi^\star,\varphi)$. A sufficient condition for this is that for some $\delta_1,\delta_2$, the class $\mathcal{D}:=\{z\mapsto \Delta_{n_2}(\eta | \varphi)(z):\|\eta-\eta_\varphi^\star\|\le\delta_1,\varphi\in\Phi_{\delta_2},\eta\in\mathcal{E}\}$ is $P$-Donsker. Assumption \ref{ass:expand}(ii) requires continuity of the map $\eta\mapsto J(\eta| \varphi)$, which is satisfied if $M_{n_2}(\eta | \varphi)$ is twice continuously differentiable in $\eta$ for each $\varphi\in\Phi_\delta$.  Assumption \ref{ass:expand}(iii) gives control on the remainder term, and will be satisfied, for instance, when  $M_{n_2}(\eta | \varphi)$ is twice continuously differentiable.
\end{remark}	

The above assumptions allow us to study the large sample behavior of the translated posterior $\pi(\eta-\eta^\star_\varphi|\bw,\varphi)$. %Intuitively, Assumption \ref{ass:ident_M}-\ref{ass:prior} gives posterior concentration of $\pi(\eta-\eta^\star_\varphi\mid \bw,\varphi)$ for each $\varphi\in\Phi_\delta$. If in addition, Assumption \ref{ass:expand} is satisfied, the dominant term in the conditional cut posterior will be a quadratic form in $(\eta-\eta^\star_\varphi)$, which will then ensure that $\pi(\eta-\eta^\star_\varphi\mid \bw,\varphi)$ is asymptotically Gaussian.
To present this behavior as succinctly as possible, define $Z_{n_2}(\eta^\star_\varphi|\varphi):=J(\eta^\star_\varphi|\varphi)^{-1}\Delta_{n_2}(\eta^\star_\varphi|\varphi)/\sqrt{n_2},$ as well as the local parameter $t:=\sqrt{n_2}(\eta-\eta^\star_\varphi)$ and its posterior $
{\pi}_{}(t|\bw,\varphi)=\pi_{}(\eta^\star_{\varphi}+t/\sqrt{n_2}|\bw,\varphi)/\sqrt{n_2}^{d_\eta},$ which has support where $\mathcal{E}_{n_2}:=\{t=\sqrt{n_2}(\eta-\eta^\star_\varphi):\eta\in\mathcal{E},\varphi\in\Phi_\delta\}$.
\begin{theorem}\label{prop:bvm1}
	If for some $\delta>0$, Assumptions \ref{ass:ident_L}-\ref{ass:expand} are satisfied for $\varphi\in\Phi_\delta$, then
	$
	\int_{\mathcal{E}_{n_2}}\|t\| \left|\pi_{}(t |\bw,\varphi)-N\{t;Z_{n_2}(\eta^\star_\varphi|\varphi),[\nu J(\eta^\star_\varphi|\varphi)]^{-1}\}\right|\dt t=o_{p}(1).
	$
\end{theorem}

Theorem \ref{prop:bvm1} demonstrates that in large samples $\pi_{}(\eta|\bw,\varphi)$ behaves like a Gaussian density with a mean and variance \textit{that both depend on $\varphi$}. This result is useful for at least two reasons. Firstly, the only other result on the behavior of cut posteriors of which we are aware, \cite{pompe2021asymptotics},  demonstrates that in large samples the cut posterior for $\theta=(\varphi^\top,\eta^\top)^\top$ is Gaussian with a variance \textit{that depends on fixed $\varphi^\star$ and $\eta^\star=\eta^\star_{\varphi^\star}$}. {(Corollary 1 in the supplementary material gives a similar result for the case of generalized posteriors.)} That is, in a conventional multivariate normal (Laplace) approximation of the joint cut posterior, the induced conditional posterior approximation results in a covariance matrix \textit{that does not depend} on the conditioning value $\varphi$ but only on $\varphi^\star$. {Since in small-to-medium sample sizes  the conditional posterior $\pi(\eta| \bw,\varphi)$ will have a mean and variance that changes with the value of $\varphi$, such a global approximation is unlikely to be accurate.}

Secondly, the conditional approximation in Theorem \ref{prop:bvm1} can be directly used in cases where accessing $\pi_{}(\eta|\bw,\varphi)$ may be difficult but where $Z_{n_2}(\eta^\star_\varphi|\varphi)$ and $J(\eta^\star_\varphi|\varphi)$ can be easily estimated. The latter may occur, for example, in cases where the MCMC sampler has a difficult time sampling $\pi_{}(\eta|\bw,\varphi)$ at the particular value of $\varphi$ on which we are conditioning. While direct use of Theorem \ref{prop:bvm1} involves 
a computational approximation to the actual conditional posterior distribution, 
the normal approximation can also useful
as a proposal distribution for MCMC or importance sampling.

%While we argue that the conditional view of the posterior for $\eta$ presented
%in Theorem \ref{prop:bvm1} is most appropriate, it is feasible to
%obtain a large sample result for the joint cut posterior.  A result along
%these lines is given in Appendix C, and generalizes the results obtained
%by \cite{pompe2021asymptotics} to a generalized Bayes setting.  

\subsection{Calibration of learning rates}

The uncertainty quantification of the generalized cut posterior
density $\pi_{\text{cut}}(\theta|\bz,\bw)$ depends crucially
on the choice of learning rates, which we now discuss.
Consider the loss likelihood term in \eqref{fullposterior}, where 
$\nu$ and $\nu'$ need to be chosen.  
\cite{lyddon+hw19}, inspired by an earlier method
of \cite{holmes+w17}, suggest to choose learning rates by matching the
Fisher information number for the generalized Bayes
update to the Fisher information number from an update based on 
a loss likelihood bootstrap approach, asymptotically.  
We do not describe here in detail the reasoning behind the
method of \cite{lyddon+hw19}, but the key to its application
here for estimation of multiple learning rates is to exploit
the modular structure of the model.   We 
set the first learning rate $\nu$ based on the prior
to posterior update for $\varphi$ in the first module, and set
the second learning rate $\nu'$ based on the conditional prior
to conditional posterior update for $\eta$ in the second module, fixing $\varphi$
to an estimate based on module one.

To state the idealized learning rates we require some additional notation.  Let
$$\Sigma_{11}=-\nabla^2_{\varphi \varphi}\mathbb{L}(\varphi^*),\;\;\;
\Sigma_{22}=-\nabla^2_{\eta \eta}\mathbb{M}(\eta^*|\varphi^*),\;\;\;
\Sigma_{12}=\nabla^2_{\varphi \eta}\mathbb{M}(\eta^*|\varphi^*),$$
$$\Psi_{11}=\lim_{n\rightarrow \infty} \text{Cov}(L_{n_1}(\varphi^*)/\sqrt{n_1}),\;\;\;
\Psi_{22}=\lim_{n\rightarrow \infty} \text{Cov}(\Delta_{n_2}(\eta^*|\varphi^*)/\sqrt{n_2}).$$
With this notation, if we apply the method of \cite{lyddon+hw19}
for choosing $\nu$ based on the update 
for the parameter $\varphi$ using the first module only, 
we obtain the ideal choice
$$\nu={\text{tr}(\Sigma_{11} {\Psi_{11}}^{-1} \Sigma_{11})}/{\text{tr}( \Sigma_{11})}.$$ 
We can estimate $\Sigma_{11}$ by ${n_1}^{-1}\nabla^2_{\varphi \varphi}L_{n_1}(\widehat{\varphi})$, where $\widehat{\varphi}=\arg\max_\varphi L_{n_1}(\varphi)$.  
To estimate $\Psi_{11}$, we could use ${n_1}^{-1}\sum_{i=1}^{n_1}
\nabla_{\varphi}\ell(z_i;\widehat{\varphi}){\nabla_{\varphi}\ell(z_i;\widehat{\varphi})}^\top$, although $\Psi_{11}$ can also be estimated in other ways.   

After calibrating $\nu$ based on the first module, we can calibrate
$\nu'$ by considering a conditional update of our beliefs for
$\eta$ in the second module, conditional on an estimate of
$\varphi$ from the first module, $\varphi=\widehat{\varphi}$ say.  
Matching the Fisher information number suggests
choosing $\nu'$ as 
$\nu'={\text{tr}(\Sigma_{22}{\Psi_{22}}^{-1} \Sigma_{22})}/{\text{tr}(\Sigma_{22})}.$
To estimate $\Sigma_{22}$ we can use
${n_2}^{-1}\nabla^2_{\eta \eta}M_{n_2}(\widehat{\eta}_{\widehat{\varphi}}|\widehat{\varphi})$,
where $\widehat{\eta}_{\widehat{\varphi}}=\arg\max_{\eta} M_{n_2}(\eta|\widehat{\varphi})$,
and $\Psi_{22}$ can be estimated by
${n_2}^{-1}\sum_{i=1}^{n_2}
\nabla_{\eta}m(w_i;\widehat{\eta}_{\widehat{\varphi}},\widehat{\varphi})
{\nabla_{\eta}m(w_i;\widehat{\eta}_{\widehat{\varphi}},\widehat{\varphi})}^\top$, or
using some other method.

In a conventional generalized Bayesian analysis, there is only one learning
rate to choose, but here there are two.  This makes choosing learning rates
more difficult, but also makes the modular generalized Bayesian approach
more flexible.  The way that marginal inferences about $\varphi$ and conditional
inferences for $\eta$ given $\varphi$ can be done separately in a modular
approach for two different loss functions 
makes the choice of two learning rates feasible. {We thank two anonymous
referees for their insight in encouraging us to explore further the choice
of separate learning rates for different modules.}

%\subsection{Discussion and Diagnostics}
\subsection{Diagnostics for $\eta|\bw,\varphi$: Understanding Uncertainty Propagation}\label{sec:diag1}
Theorem \ref{prop:bvm1} demonstrates that even in large samples the behavior of $\pi(\eta|\bw,\varphi)$ depends on the value of $\varphi$ on which we are conditioning. Moreover, for different values of $\varphi$, the resulting mean and variance can vary substantially. It is therefore useful to understand how our uncertainty about $\varphi$  propagates into our inferences for $\eta$.

Using the result of Theorem \ref{prop:bvm1}, this uncertainty can be viewed in many different ways. For instance, if $\eta$ is low-dimensional, we can visualise the impact of $\varphi$ on the posterior for $\eta|\bw,\varphi$ by viewing the kernel $$|J(\eta^\star_\varphi|\varphi)|^{1/2}\exp\left\{-{{n_2}\cdot\nu\cdot}(\eta-\eta^\star_\varphi)^\top J(\eta^\star_\varphi|\varphi)^{}(\eta-\eta^\star_\varphi)/2\right\},$$ across a given range of values for $\varphi$. The resulting plot will demonstrate how the cut posterior for $\eta$ changes as the conditioning value of $\varphi$ changes. 

The above approximation cannot be directly accessed, since $\eta^\star_\varphi$ and $J(\eta^\star_\varphi|\varphi)$ are unknown in practice. However,  in cases where $M_{n_2}(\eta|\varphi)$ is twice continuously differentiable in $\eta$, it is simple to estimate $\eta^\star_\varphi$ and $J(\eta^\star_\varphi|\varphi)$ by their empirical counterparts  $\widehat\eta_\varphi:=\argmax_{\eta}M_{n_2}(\eta|\varphi)$, and  $J_{n_2}(\widehat\eta_\varphi|\varphi):={n_2}^{-1}\nabla^2_{\eta\eta}M_{n_2}(\widehat{\eta}_\varphi |\varphi)$ respectively.

The large sample approximation can also be used to visualize the behavior of specific functionals of interest, e.g., moments of $\eta|\bw,\varphi$. 
As an example, suppose that $\eta$ is a scalar for simplicity and 
that we are interested in understanding how the variance of its 
cut posterior depends on $\varphi$.  Using
the law of total variance, we can write
$$\text{Var}(\eta)=E(\text{Var}(\eta|\varphi))+\text{Var}(E(\eta|\varphi)),$$
(where expectations in this expression are with respect to the cut
posterior) and for draws $\varphi^{(s)}\sim \pi_{\text{cut}}(\varphi|\bz)$, 
$s=1,\dots, S$, we can plot histograms of $\text{Var}(\eta|\varphi^{(s)})$
and $E(\eta|\varphi^{(s)})$ to understand 
how variability in $\eta$ relates to $\varphi$.  
The conditional means and variances can be approximated by the
normal approximations obtained from Theorem 1.  In an example 
in Section 4.1, we discuss diagnostics of this type, as well 
as methods for understanding posterior skewness in the parameter 
in the second module.

Credible sets of $\pi_{}(\eta|\bw,\varphi)$ are also of particular interest. For $\alpha\in(0,1)$, let $C^\eta_\alpha(\varphi)$ be the set of $\eta$ such that 
$
\int_{C^\eta_\alpha(\varphi)}\pi_{}(\eta|\bw,\varphi)\dt\eta=1-\alpha.
$ Then, we can visualize $C^\eta_\alpha(\varphi)$ across several values of $\varphi\in\Phi$ to understand how the shape of credible sets change as $\varphi$ varies. In the case of credible sets, the normal approximation can be directly used to obtain an estimate of $C^\eta_\alpha(\varphi)$, and 
an algorithm for this is given in Appendix C.

Without Theorem \ref{prop:bvm1}, constructing functionals of $\eta|\bw,\varphi$ at different values of $\varphi$ usually requires running an MCMC sampling algorithm to obtain draws of $\eta|\bw,\varphi$.  Theorem \ref{prop:bvm1} gives a computationally cheap alternative: we simply replace $\pi_{}(\eta|\bw,\varphi)$ in the definition of the functional by the normal approximation in Theorem \ref{prop:bvm1}, for which samples can be drawn directly.

\subsection{Incorporating Feedback via  Tempering}\label{sec:semimod}
Recently, \cite{carmona2020semi} have proposed the use of ``semi-modular" posterior distributions as an extension of cut-model inference; \cite{nicholls2022valid} extend this construction to prequentially additive loss functions, and \cite{carmona+n22} investigate the use of normalizing flows for their computation.   
In this section we explain semi-modular inference and 
introduce a new type of semi-modular posterior, for which
fast computation is possible using the asymptotic approximations developed in Section 3.1.  %However, sampling from the semi-modular posterior can be computationally challenging. %In this section, we produce a computationally simple, but accurate, approximation to the semi-modular posterior. 

Consider once again the two module system, and 
point estimation for the shared parameter $\varphi$ based on 
full and cut posterior distributions.  
The intuition behind semi-modular inference is that if the degree of 
misspecification is not severe, then the bias of the full posterior estimator
may only be moderate, while its variance
might be greatly reduced compared to the cut posterior estimator.  In this case,  
full posterior estimates may have better frequentist
performance in managing a bias-variance
trade-off.    If the misspecification is serious, 
however, full posterior estimation may have a large bias, and 
estimation based on the cut model may be preferred.
Instead of making a binary choice between the full and cut 
posterior density, it might be better to modulate
the influence of the misspecified module in a more continuous
way, using an ``influence parameter" denoted here as $\gamma\in [0,1]$.  
In the proposal of \cite{carmona2020semi}, the choice $\gamma=0$ 
results in the cut posterior, whereas $\gamma=1$ 
corresponds to the full posterior, so that the semi-modular posterior interpolates
between cut and full posterior based on the influence parameter.  
\cite{nicholls2022valid} also explore
some more Bayesian properties of 
validity and order-coherence of semi-modular posteriors 
for their original approach and some alternatives.  

The semi-modular method of \cite{carmona2020semi} proceeds in two stages. First, an auxiliary parameter $\tilde\eta$ is introduced that replicates the role of $\eta$ in the second module.  Extending the discussion of \cite{carmona2020semi} to the generalized Bayes setting, 
they would construct a posterior density for $(\varphi,\tilde\eta)$ as
\begin{align}
	\pi_{\text{pow},\gamma}(\varphi,\tilde\eta|\bz,\bw) \propto 
	\pi(\varphi)\pi(\tilde\eta|\varphi)\exp\left\{\nu L_{n_1}(\varphi)\right\}\exp\left\{\nu M_{n_2}(\tilde\eta|\varphi)\right\}^\gamma, \label{CN-auxiliary}
\end{align}
where $\gamma\in [0,1]$ is an influence parameter which controls how much
of the information in the second module is used in making marginal
inferences about $\varphi$.  A joint density for $(\varphi,\tilde\eta, \eta)$
is then constructed by multiplying \eqref{CN-auxiliary} by the conditional
posterior density for $\eta|\varphi$, followed by integrating out $\tilde\eta$, 
to obtain the semi-modular posterior:
\begin{align}
	\pi_\gamma^{\text{CN}}(\varphi,\eta|\bz,\bw)=\int \pi_{\text{pow},\gamma}
	(\varphi,\tilde\eta|\bz,\bw)d\tilde\eta \times \pi(\eta|\varphi,\bw). \label{CN-smi}
\end{align}
It is easy to see that if $\gamma=0$, \eqref{CN-smi} is the cut posterior
density, whereas $\gamma=1$ gives the conventional joint posterior
density.  

\subsubsection{Marginal semi-modular inference}

We now introduce an alternative semi-modular approach, where no auxiliary parameter $\tilde{\eta}$ is introduced and computation can be conveniently
done using the large sample approximations developed in Section 3.1.   Recall that the marginal generalized posterior for $\varphi$ can be written as $\pi_{}(\varphi|\bz,\bw)=\int_{\mathcal{E}}\pi(\varphi,\eta|\by)\dt\eta$. Using the decomposition of $\pi(\varphi,\eta|\by)$ in equation \eqref{eq:decomp}, and the fact that $\int_{\mathcal{E}}\pi(\eta|\bw,\varphi)\dt\eta=1$ for each $\varphi\in\Phi$, we can rewrite $\pi_{}(\varphi|\bz,\bw)$ as 
$$
\pi_{}(\varphi|\bz,\bw)=\pi_{\cut}(\varphi|\bz)\tilde{p}(\bw|\varphi)\propto\pi_\cut(\varphi|\bz){m_\eta(\bw|\varphi)},
$$where $m_\eta(\bw|\varphi)$ was defined in equation \eqref{eq:post}.% and $\int_\Phi \pi_{\cut}(\varphi|\bz)\tilde{p}(\bw|\varphi)\dt\varphi=1$ since $\int_{\Theta}\pi(\theta|\by)\dt\theta=1$ by definition. 

The above suggests that a type of marginal semi-modular inference for $\varphi$ can proceed via the tempered marginal posterior
\begin{align}
	\pi^{\mathrm{M}}_\gamma(\varphi|\bz,\bw) & \propto\pi_{\cut}(\varphi|\bz)\tilde{p}(\bw|\varphi)^\gamma\propto\pi_\cut(\varphi|\bz){m_\eta(\bw|\varphi)}^\gamma,\;\gamma\in[0,1],\label{smi1}
\end{align} where the notation `M' makes clear that we are only considering a marginal semi-modular posterior.  
The marginal semi-modular posterior $\pi^{\mathrm{M}}_\gamma(\varphi|\bz,\bw)$ attenuates the impact of the feedback term $\tilde{p}(\bw|\varphi)$ by tempering its contribution in the marginal posterior, interpolating between the cut posterior marginal $\pi_{\text{cut}}(\varphi|\bz)$ at $\gamma=0$, and the conventional marginal posterior $\pi(\varphi|\bz,\bw)$ at $\gamma=1$.  

\subsubsection{Computation for marginal semi-modular inference}

The difficulty in computing  $\pi^{\mathrm{M}}_\gamma(\varphi|\bz,\bw)$  lies in calculating $\tilde{p}(\bw|\varphi)$, which is intractable in cases where $m_\eta(\bw|\varphi)$ is intractable. However, observe that the form of $m_\eta(\bw|\varphi)$ is akin to that of a ``marginal likelihood'' conditioned on a fixed $\varphi$. Indeed, $m_\eta(\bw|\varphi)$ satisfies the following tautological relationship:
\begin{align}
	m_\eta(\bw|\varphi) & ={\pi(\eta|\varphi)\exp\{\nu M_{n_2}(\eta|\varphi)\}}/{\pi(\eta|\bw,\varphi)}, \label{tautology}
\end{align}
which holds for any values of $\eta$ and $\varphi$. The above equation resembles the ``basic marginal likelihood identity'' used in \cite{chib1995marginal} to estimate the marginal log-likelihood, and is related to the ``candidate's formula'' presented in \cite{besag1989candidate}. Following \cite{chib1995marginal}, taking logarithms of \eqref{tautology} and considering
a chosen value $\eta^\ast$ of $\eta$ in the high probability region for
$\pi(\eta|\bw,\varphi)$, we obtain 
\begin{eqnarray}\label{eq:marg}
	\ln m_\eta(\bw|\varphi)=\ln\pi(\eta^\ast|\varphi)+\nu M_{n_2}(\eta^\ast|\varphi)-\ln \pi(\eta^\ast|\bw,\varphi).
\end{eqnarray}
Given a choice of $\eta^\ast$, replacing
$\ln \pi(\eta^\star|\bw,\varphi)$ with an estimate of it in \eqref{eq:marg}, results
in an estimate $\ln\widehat{m}_\eta(\bw|\varphi)$ of $\ln m_\eta(\bw|\varphi)$.  %However, it is important to realize that, from the tautological relationship for $m_\eta(\bw|\varphi)$ expressed above, the above estimator specific point $\eta^\star$ 

If we have a known form for $\pi(\eta |\bw,\varphi)$, then this form can be directly used to obtain the estimate $\ln\widehat{m}_\eta(\bw|\varphi)$. However, in general,  $\pi(\eta |\bw,\varphi)$ is not available in closed form, and there are two obvious approaches: one, we use posterior draws and kernel density estimation to estimate $\pi(\eta^\star |\bw,\varphi)$; two, we use the large sample approximation to $\pi(\eta |\bw,\varphi)$ obtained in Theorem \ref{prop:bvm1} at the point $\eta^\star$. 

The latter approach is simple to implement  when $M_{n_2}(\eta|\varphi)$ is twice-continuously differentiable in $\eta$, for all $\varphi\in\Phi$. In this case, we obtain from \eqref{eq:marg} the estimate:
\begin{eqnarray}\label{eq:margest}
	\ln\widehat{m}_\eta(\bw|\varphi)=\ln\pi(\eta^\star|\varphi)+\nu M_{n_2}(\eta^\star|\varphi)-\ln N\{\eta^\star;\hat\eta_\varphi%+Z_{n_2}(\widehat\eta_\varphi|\varphi)/ {n}
	,[n_2\nu J_{n_2}(\widehat\eta_\varphi|\varphi)]^{-1}\},
\end{eqnarray}where we recall that $\widehat\eta_\varphi:=\argmax_{\eta\in\mathcal{E}}M_{n_2}(\eta|\varphi)$, and $J_{n_2}(\widehat\eta_\varphi|\varphi):={n_2}^{-1}\nabla^2_{\eta\eta}M_{n_2}(\widehat{\eta}_{\varphi}|\varphi)$. Critically, unlike $m_\eta(\bw|\varphi)$, up to the calculation of $\widehat\eta_\varphi$ and $J_{n_2}(\widehat{\eta}_\varphi|\varphi)$, the estimator in \eqref{eq:margest} is known in closed-form.

Given  $\ln\widehat{m}_\eta(\bw|\varphi)$, and fixed $\gamma\in[0,1]$, we can use 
\begin{eqnarray*}
	\widehat\pi^{\mathrm{M}}_\gamma(\varphi|\bz,\bw) & \propto\pi_\cut(\varphi|\bz)[\exp\{\ln\widehat{m}_\eta(\bw|\varphi)\}]^\gamma,
\end{eqnarray*}
as an approximation to the semi-modular posterior $\pi^{\mathrm{M}}_\gamma(\varphi|\bz,\bw)$. For fixed $\gamma$, the approximate semi-modular posterior $\widehat\pi^{\mathrm{M}}_\gamma(\varphi|\bz,\bw)$ can be sampled using a Metropolis-Hastings MCMC (MH-MCMC) algorithm. However, since $\ln\widehat{m}_\eta(\bw|\varphi)$ is only an estimate of $m_{n_2}(\bw|\varphi)$, the resulting MH-MCMC algorithm will not deliver draws from $\pi^{\mathrm{M}}_\gamma(\varphi|\bz,\bw)$. Nonetheless, in large samples, by Theorem \ref{prop:bvm1}, we can expect these draws to yield an accurate approximation to $\pi^{\mathrm{M}}_\gamma(\varphi|\bz,\bw)$. 
Similar to the method of \cite{carmona2020semi}, 
the choice of the influence parameter $\gamma$ in our
approach can be carried out using predictive approaches similar to those in \cite{carmona2020semi}, or using the conflict checks, see \cite{chakraborty2023modularized}.

It is interesting to compare our new approach to semi-modular
inference based on \eqref{smi1} with the method of \cite{carmona2020semi}, 
and we give an empirical comparison for the example of Section 4.1 in the supplementary material, finding
they give similar results.  Statistically there seems no clear reason
to prefer one approach over the other in the examples we 
have considered, and a more detailed theoretical study is left
for future work.   However, 
the computational approximations based on Theorem 1 are 
helpful for implementing both of these methods.  The method of \cite{carmona2020semi} is often implemented using a nested MCMC 
approach similar to
that used in cut posterior computation.  So repeated sampling of
$\eta$ given $\varphi$ for samples $\varphi$ from the SMI marginal
posterior for $\varphi$ can be done cheaply 
using the normal approximations justified by Theorem 1, 
or these approximations can be used as a good proposal to accelerate
MCMC or importance sampling.  
In the supplementary material, Algorithm \ref{alg:divs} describes how to sample from the proposed semi-modular posterior \eqref{smi1} when $\ln \widehat{m}_\eta(\bw|\varphi)$ is obtained using \eqref{eq:margest}.

\section{Examples}\label{sec:examples}
%In this section, we demonstrate the diagnostics proposed in Section \ref{sec:diag1} and the semi-modular approach in Section \ref{sec:semimod} in 
In this section we consider 
two examples.  The first example illustrates our 
large sample approximations for cut posterior computation, 
and for implementing diagnostics
for understanding uncertainty propagation between modules.  
We consider both probabilistic model specifications as well as a
generalized Bayesian analysis using a quasi-likelihood.
Our second example also considers
a generalized Bayesian analysis, for which the learning rate for the second
module needs to be carefully chosen.  We illustrate a situation where an
appropriate choice of the loss function can resolve conflict between 
cut and full posterior inferences, giving insight into how an initially flawed
parametric model may need to be improved.  

\subsection{HPV prevalence}\label{sec:hpv}
Our first example was discussed in \cite{plummer2015cuts},
and is based on a real epidemiological study \citep{maucort2008international}.  
The model consists of two modules.  Module 1 incorporates survey data from 13 countries on
high-risk human papillomavirus (HPV) prevalence for women in a certain age group.  
Denote by $z_i$ the number of women with high-risk HPV in country $i$ in a survey of $N_i$ individuals, $i=1,\dots, 13$, and assume that 
$z_i\sim \text{Binomial}(N_i,\varphi_i)$, where $\varphi_i\in [0,1]$ is a country-specific prevalence probability.  
The parameters $\varphi_i$ are assumed independent in their prior, with $\varphi_i\sim U[0,1]$.  Write
$\varphi=(\varphi_1,\dots, \varphi_{13})^\top$.

Module 2 incorporates cervical cancer incidence data
$\bw$, with $w_i$  the number of cervical cancer cases in 
$T_i$ woman years of follow-up in country $i$, $i=1,\dots, 13$.  
The relationship between cervical cancer incidence and HPV prevalence is described by a Poisson regression model, 
$w_i\sim \text{Poisson}(T_i \rho_i)$, where $\log \rho_i = \eta_1+\eta_2 \varphi_i$.  
For these data the Poisson regression model is misspecified, and because $\varphi_i$ is appearing as a covariate
in the Poisson regression, inference about $\varphi_i$ is influenced by the misspecification in the second module.  
Estimation of $\varphi$ adapts to the misspecification, distorting inference about these parameters, which also
results in uninterpretable inference about the regression parameters $\eta$ used to summarize the relationship
between HPV prevalence and the rate of cancer incidence.   The main interest of the analysis lies in understanding
this relationship.

\subsubsection{Cut posterior computation with large sample approximation}

When cutting feedback in this example, it is straightforward to obtain posterior samples from $\pi_{\text{cut}}(\varphi|\bz)$. This is because the likelihood for each $z_i$ is binomial, 
and the priors for the parameters $\varphi_i$ are conjugate.  In $\pi_{\text{cut}}(\varphi|\bz)$ the $\varphi_i$ are independent, with 
$\pi_{\text{cut}}(\varphi_i|\bz)$ a beta density, $\text{Beta}(z_i+1,n_i-z_i+1)$.  We generate samples $\varphi^{(s)}$, $s=1,\dots, S=1000$,
from $\pi_{\text{cut}}(\varphi|\bz)$ by direct Monte Carlo sampling.  
To generate samples $\eta^{(s)}$ so that $(\varphi^{(s)},\eta^{(s)})$ is a draw
from the joint cut posterior density, we do the following.  By Theorem \ref{prop:bvm1}, we can approximate
the conditional posterior density of $\eta$ given $\varphi^{(s)}$, $w$ and $z$ by a normal density with mean
$\mu(\varphi^{(s)})=\widehat{\eta}_{\varphi^{(s)}}$ and covariance matrix $\Sigma(\varphi^{(s)})=n^{-1}\nabla_{\eta \eta} M_n(\widehat{\eta}_{\varphi^{(s)}}|\varphi^{(s)})$. %\left|_{\eta=\widehat{\eta}_{\varphi^{(s)}}}\right.$ 
For each $\varphi^{(s)}$, we generate $1,000$ proposal samples for $\eta$ from a multivariate $t$-distribution with mean $\mu(\varphi^{(s)})$, scale matrix $\Sigma(\varphi^{(s)})$, and $5$ degrees of freedom, and draw a single sample $\eta^{(s)}$ 
from these proposals using sampling importance resampling (SIR).   

For comparison, we can also draw an approximate sample
$\widetilde{\eta}^{(s)}$ say from the conditional normal approximation directly.   For practical purposes the SIR samples can
be considered near-exact, and Figure \ref{exact-approximate} (top row) shows the marginal posterior samples for 
$\eta=(\eta_1,\eta_2)$ for the two approaches.  
A sample based estimate of the $1$-Wasserstein distance between the posterior marginal cut distributions 
estimated by the exact SIR and approximate conditional normal methods is $0.004$ and $0.062$ for $\eta_1$ and $\eta_2$ respectively, showing
that our large sample conditional normal approximations 
result in accurate cut posterior computation.  We can see that the marginal cut posterior density for $\eta$ is non-Gaussian, 
but this is captured very well in the approximate sampling approach where the conditional posterior density for $\eta$ is close to
normal. It is the uncertainty about $\varphi$ that is propagated in making marginal inferences about $\eta$ that results
in the non-Gaussian structure in the marginal posterior distribution for $\eta$. 
Also shown in Figure \ref{exact-approximate} are samples from the usual Bayesian posterior distribution, obtained via MCMC using the {\texttt{rstan}} package \citep{carpenter_stan_2017}.  The full and cut posterior inferences
differ substantially, demonstrating how much the misspecification of
the second module changes the inference about $\eta$ here.
The bottom row of the figure compares the univariate marginals for the cut and full posterior densities
for $\eta_1$ and $\eta_2$.
\begin{figure}[H]
	\begin{center}
		\begin{tabular}{c}
			\includegraphics[width=80mm]{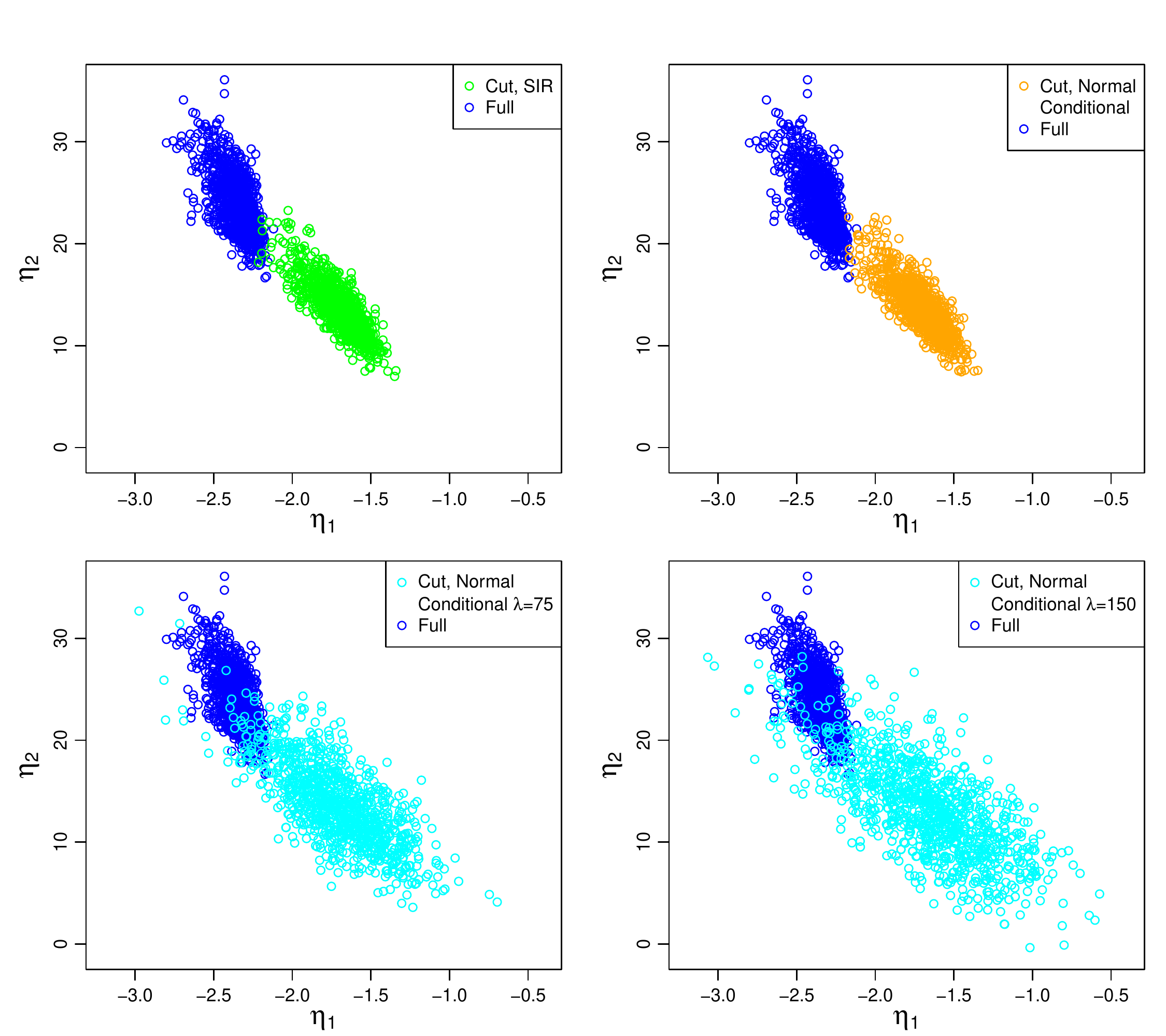}  \\
			\includegraphics[width=80mm]{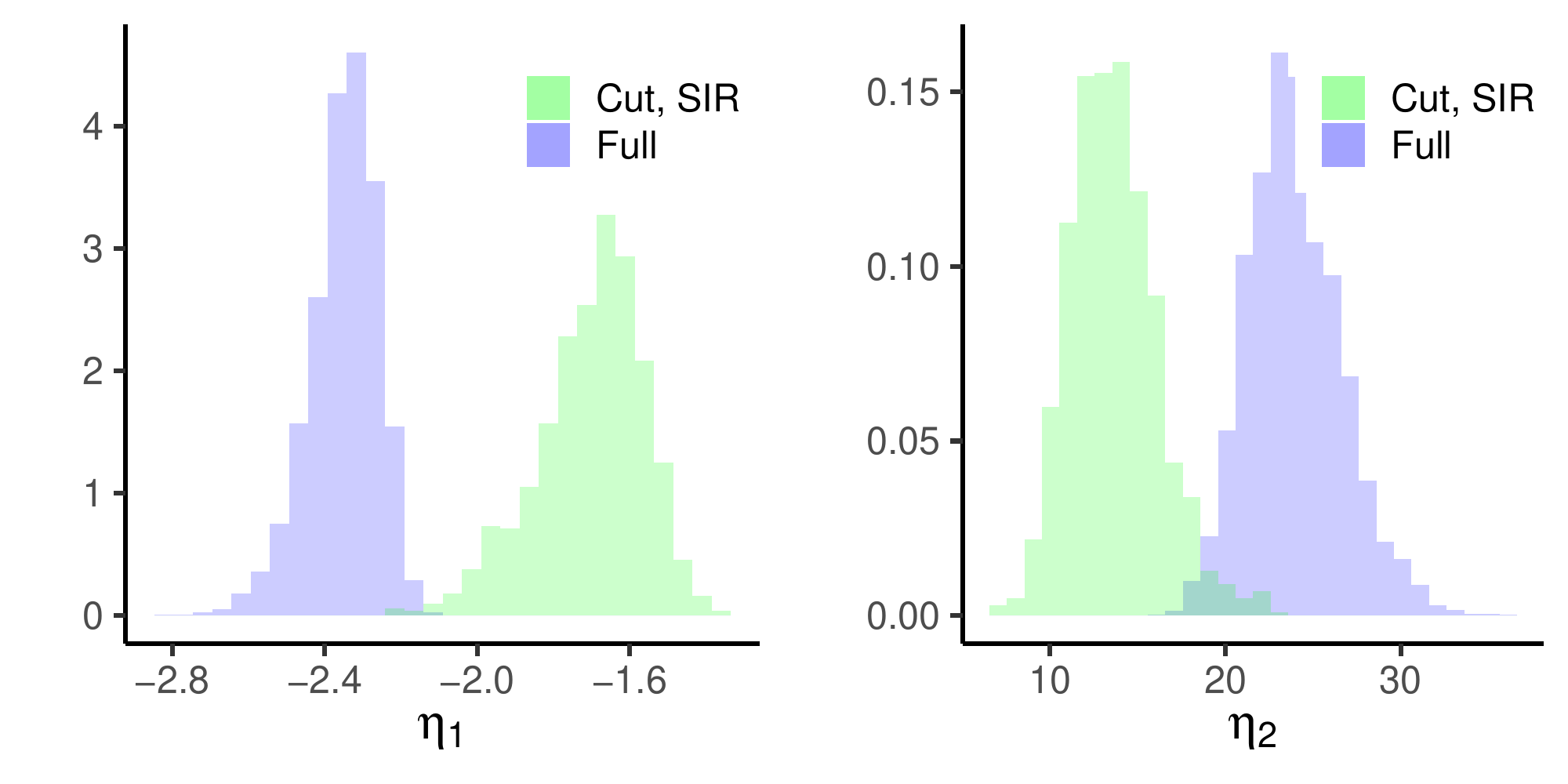}
		\end{tabular}
	\end{center}
\vspace{-.75cm}

	\caption{\label{exact-approximate}Top: marginal posterior samples for $\eta$ for full posterior density (blue) obtained by MCMC, 
		and cut posterior samples by SIR (green) and approximation by conditional normal sampling (orange). 
		Middle:  marginal posterior samples for $\eta$ for full posterior density (blue) obtained by MCMC, and cut posterior samples
		from conditional normal sampling (cyan) for quasi-likelihood loss for the second module with $\lambda=75$ (left) and
		$\lambda=150$ (right).  
		Bottom:  histogram density estimates for $\eta_1$ (left) and $\eta_2$ (right) for full posterior (blue) and cut posterior
		samples by SIR (green).}
\end{figure}

\subsubsection{Generalized posterior analysis}

The middle row of Figure \ref{exact-approximate} shows samples from the generalized cut posterior distribution 
obtained when the Poisson likelihood is replaced by a quasi-likelihood \citep{wedderburn74}, which allows for overdispersion with respect
to the Poisson model.   When using the negative 
log quasi-likelihood as the loss for the second module, 
it is sensible to choose a learning rate $\nu'=1$.  For the first module
we use the same parametric model as before.
The overdispersion parameter in the quasi-likelihood
is denoted by $\lambda$, and instead of making the Poisson
assumption that the mean and variance are equal, it is assumed that the variance is $\lambda$ times the mean
for each $w_i$.  The left plot in the middle row is for $\lambda=75$, and
the right plot is for $\lambda=150$.  We can see that even if we assume a standard deviation for the $w_i$ that 
is more than 10 times that implied by a Poisson mean-variance relationship, the full posterior samples do not become plausible
under the cut distribution.  \cite{yu2021variational} have elaborated on 
the comparison of the cut and full posterior distributions as a kind of 
conflict check, and the lack of consistency of the cut and full posterior
inferences here suggests that altering the parametric Poisson
regression to another parametric model incorporating 
multiplicative overdispersion will not result in an adequate generative
model for the data unless the degree of overdispersion is very large. 
The samples in the quasi-likelihood analysis 
were generated using the conditional normal approximation for
the density of $\eta$ given $\varphi$.  

\subsubsection{Uncertainty propagation}

Figure \ref{ellipses} shows, for 5 samples from the marginal cut posterior distribution of 
$\varphi$, a 95\% probability ellipsoid of minimal volume for
the conditional normal approximations of $p(\eta|\varphi,y)$. 
The 5 $\varphi$ samples 
are selected from $1,000$ cut posterior 
samples according to the $0.1$, $0.3$, $0.5$, $0.7$ and $0.9$ quantiles of the determinant
of the estimated conditional covariance matrix of $\eta$ given $\varphi$.
The variation in the shape of these ellipsoids is substantial as $\varphi$ changes.  
\begin{figure}[H]
	\begin{center}
		\begin{tabular}{c}
			\includegraphics[width=100mm]{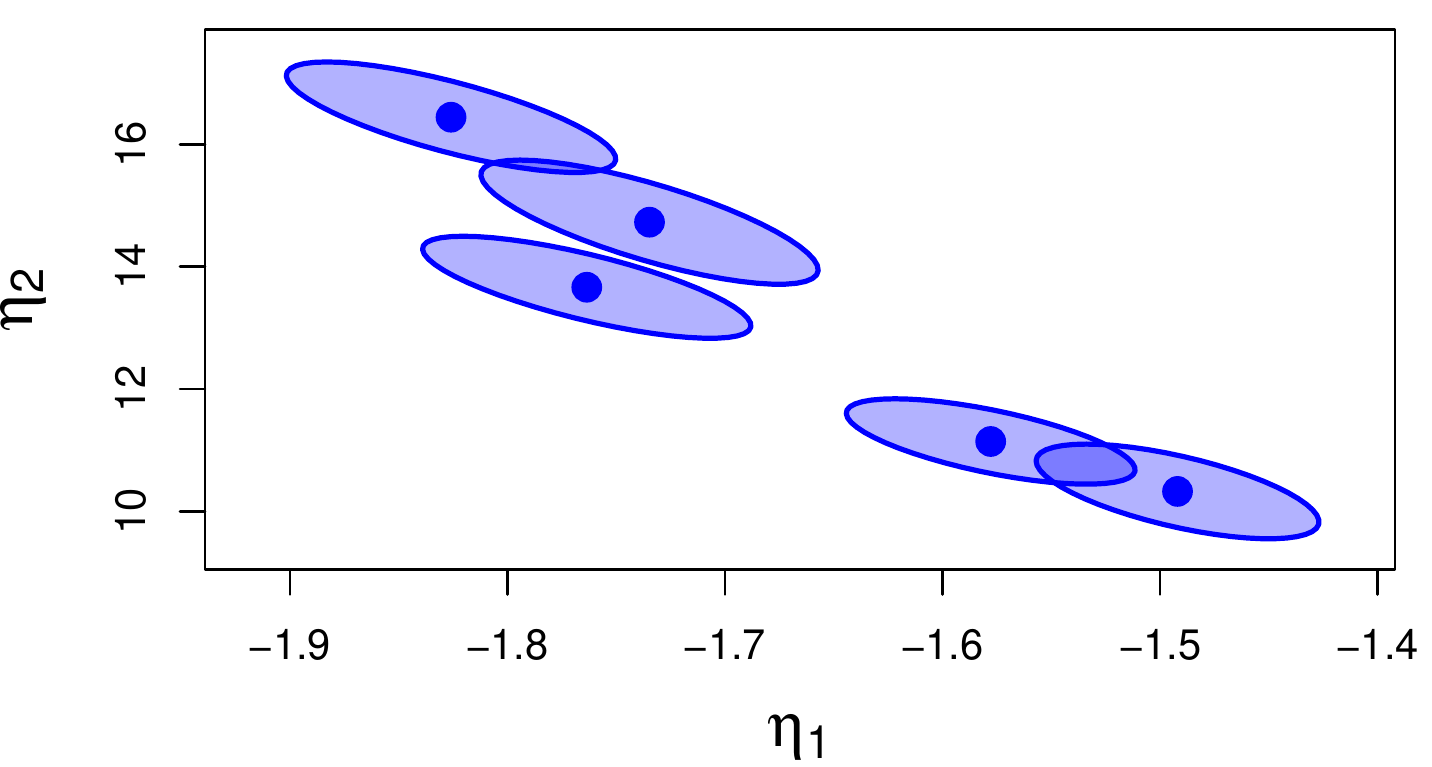}  
		\end{tabular}
	\end{center}
\vspace{-.7cm}

	\caption{\label{ellipses}95\% probability ellipsoids of minimal volume for the normal approximation to the conditional posterior
		density of $\eta$ given $\varphi$ for 5 draws from the marginal cut posterior distribution of $\varphi$. The 5 $\varphi$ samples 
		are selected from $1,000$ cut posterior 
		samples according to the $0.1$, $0.3$, $0.5$, $0.7$ and $0.9$ quantiles of the determinant
		of the estimated conditional covariance matrix of $\eta$ given $\varphi$.}
\end{figure}

We can also use the normal approximation to the conditional posterior density as a diagnostic to understand the way 
that the uncertainty in $\varphi$ propagates into the second module, for both the cut and full posterior density.
Noting that 
\begin{align}
	\text{Var}(\eta_j) & = E(\text{Var}(\eta_j|\varphi))+\text{Var}(E(\eta_j|\varphi)), \label{law-of-total-variance}
\end{align}
we could plot histograms of the values $\mu(\varphi^{(s)})_j$, $s=1,\dots, S$ and
$\Sigma(\varphi^{(s)})_{jj}$, $s=1,\dots, S$ for $j=1,2$ to understand how  uncertainty in $\varphi$ propagates
into $\eta$.  In (\ref{law-of-total-variance}) the expectations can be defined as with respect to either the full
posterior distribution or with respect to the cut posterior distribution.  
The mean of the samples in a histogram of $\Sigma(\varphi^{(s)})_{jj}$ relates to the first term on the 
right-hand side of (\ref{law-of-total-variance}).  
The variability of the samples in a 
histogram of $\mu(\varphi^{(s)})_j$ assesses 
variability propagated to $\eta_j$ from the second term on the right-hand side of  (\ref{law-of-total-variance}). 

Generalizing (\ref{law-of-total-variance}) to third central moments using the law of total cumulants 
\citep{brillinger69}, we can also write 
\begin{align}
	E((\eta_j-E(\eta_j))^3) = &  E(E((\eta_j-E(\eta_j|\varphi))^3|\varphi))+E((E(\eta_j|\varphi)-E(\eta_j))^3)+ \nonumber \\
	& \quad 3 \text{Cov}(E(\eta_j|\varphi),\text{Var}(\eta_j|\varphi)). \label{law-of-total-cumulance}
\end{align}
Once again, the expectations in the above expression can be defined as with respect to either the full
posterior distribution or with respect to the cut posterior distribution.
If the conditional posterior for $\eta_j$ given $\varphi$ is approximately symmetric, then 
the first term  on the right-hand side of (\ref{law-of-total-cumulance}) can be neglected.  Then the 
posterior skewness of $\eta_j$ depends on the second and third terms.  These terms relate to
the skewness of the conditional expectation $E(\eta_j|\varphi)$ (considered
as a function of $\varphi$) and the covariance between the conditional mean and conditional variance.  
The skewness of the conditional expectation can be assessed from looking at the skewness in a histogram
of $\mu(\varphi^{(s)})_j$, while plotting
the samples $(\mu(\varphi^{(s)})_j,\Sigma(\varphi^{(s)})_{jj})$, $s=1,\dots, S$, is helpful for 
assessing the $\text{Cov}(E(\eta|\varphi),\text{Var}(\eta|\varphi))$  term in (\ref{law-of-total-cumulance}).  

Figure \ref{diagnostics1} shows a scatterplot of $(\mu(\varphi^{(s)})_1,\Sigma(\varphi^{(s)})_{11})$, $s=1,\dots, S$, 
with histograms of each variable on the axes, for $\eta_1$.  The plot on the left is for the cut posterior density, and
the plot on the right is for the full posterior density.  There is a strong negative relationship between the conditional
posterior mean of $\varphi$ and its conditional variance, as well as negative skewness in the histogram
of $\mu(\varphi^{(s)})_{1}$, which by (\ref{law-of-total-cumulance}) explains the negative skew in the marginal distribution
for $\eta_1$ evident in Figure \ref{exact-approximate}.  This is so for both the cut and full posterior densities. 
\begin{figure}[h]
	\begin{center}
		\begin{tabular}{cc}
			\includegraphics[width=65mm]{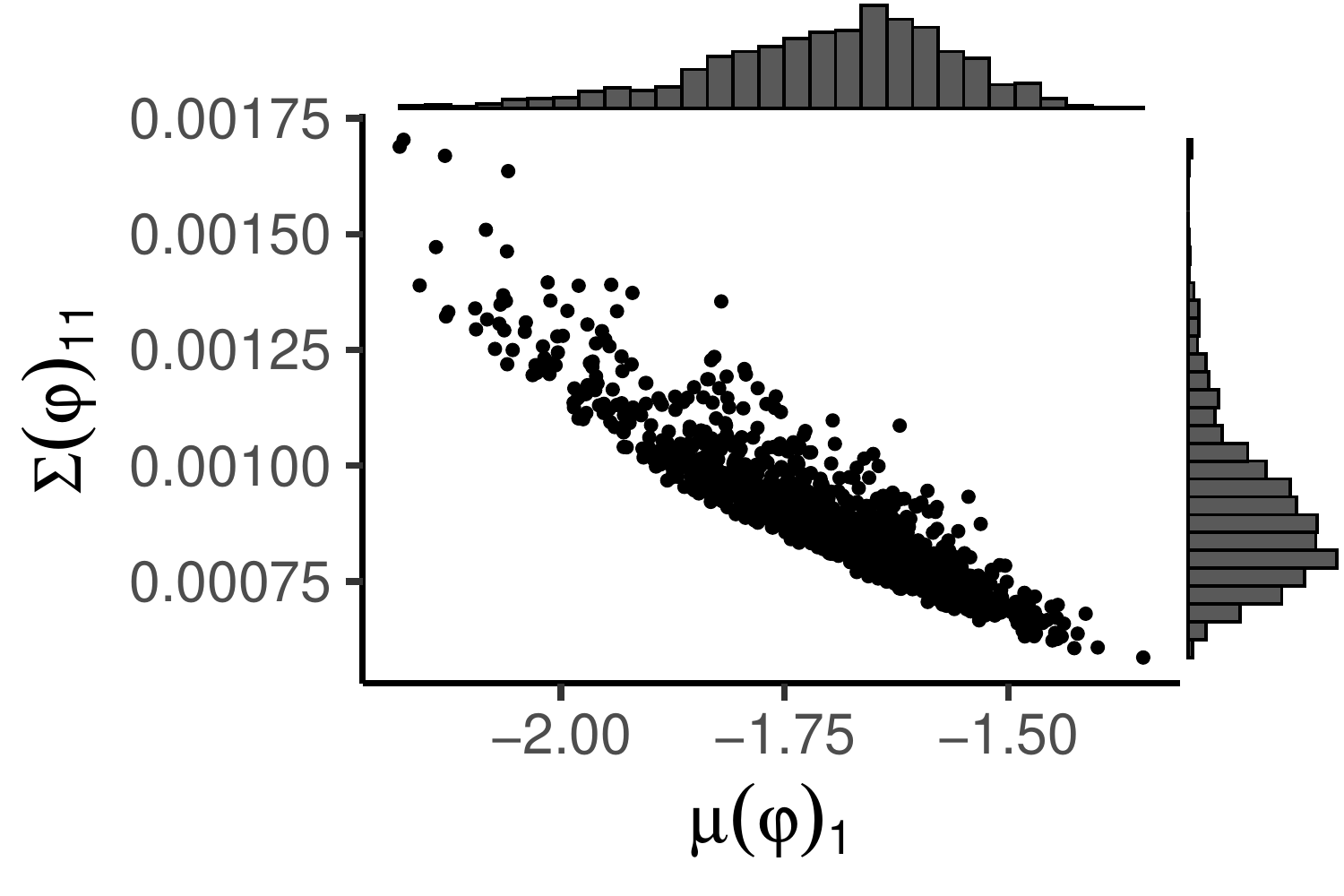} & \includegraphics[width=65mm]{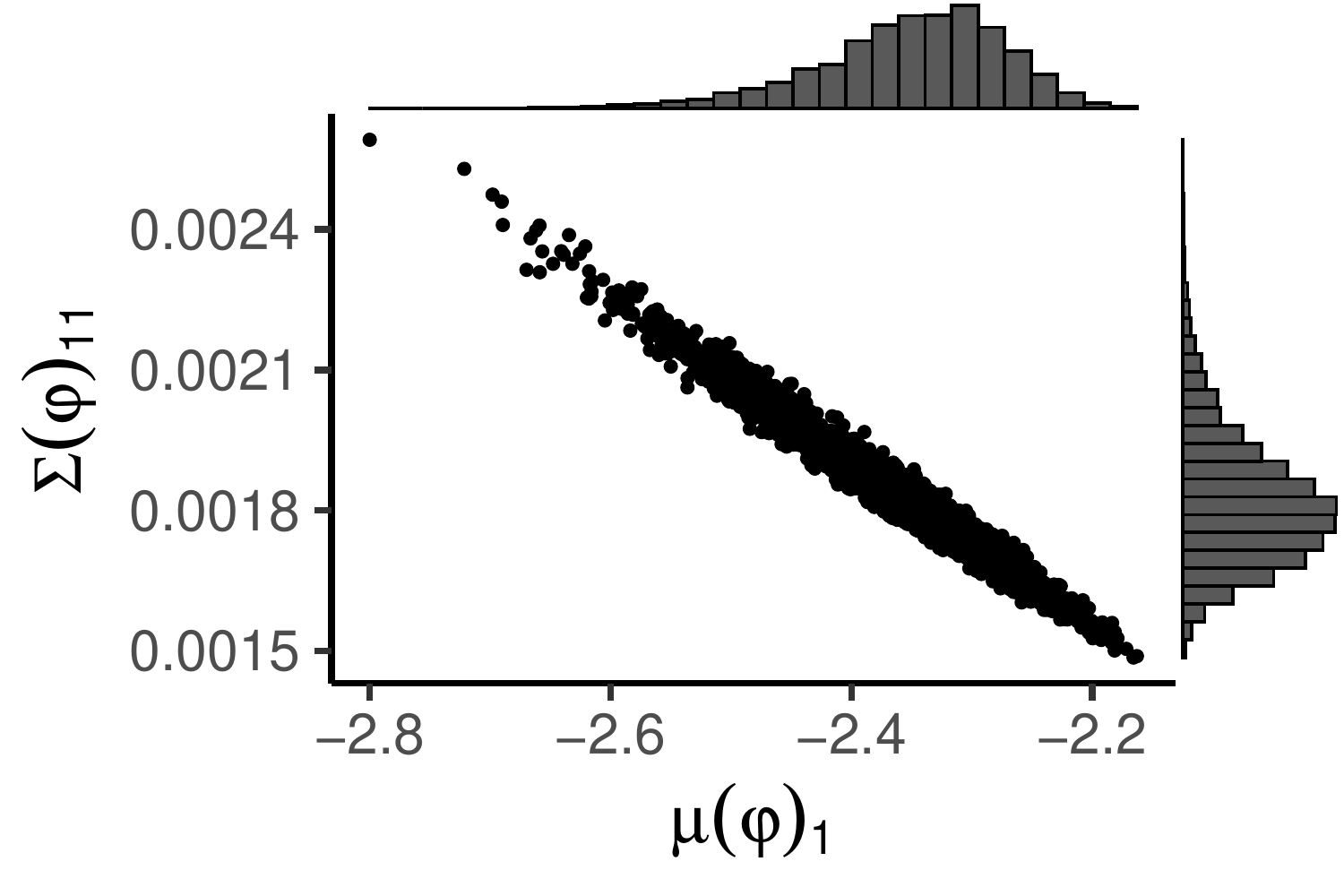}
		\end{tabular}
	\end{center}
\vspace{-.7cm}
	\caption{\label{diagnostics1}Scatterplot of $(\mu(\varphi^{(s)})_1,\Sigma(\varphi^{(s)})_{11})$, $s=1,\dots, S$, 
		for cut posterior (left) and full posterior (right) samples.  Histograms of each variable are shown on the axes.}
\end{figure}

Figure \ref{diagnostics2} shows a similar plot to Figure \ref{diagnostics1} for the parameter $\eta_2$.  
In this case, there is a strong positive relationship between the conditional posterior mean of $\varphi$ and
its conditional variance, and positive skewness in the histogram of $\mu(\varphi^{(s)})_2$, which explains
the positive skew in the marginal distribution of $\eta_2$, in both the cut and full posterior densities, as shown
in Figure \ref{exact-approximate}.
The dependence between $\mu(\varphi)_j$ and $\Sigma(\varphi)_{jj}$ in Figures \ref{diagnostics1} and
\ref{diagnostics2} relates directly to the way the conditional variance of $\eta$ depends on $\varphi$, which is
exactly what is being captured in the conditional perspective taken in the theory of Section 3.1.  
Understanding this dependence is particularly useful for explaining the the marginal  posterior shape
for $\eta$ in the full and cut posterior distributions.
\begin{figure}[h]
	\begin{center}
		\begin{tabular}{cc}
			\includegraphics[width=65mm]{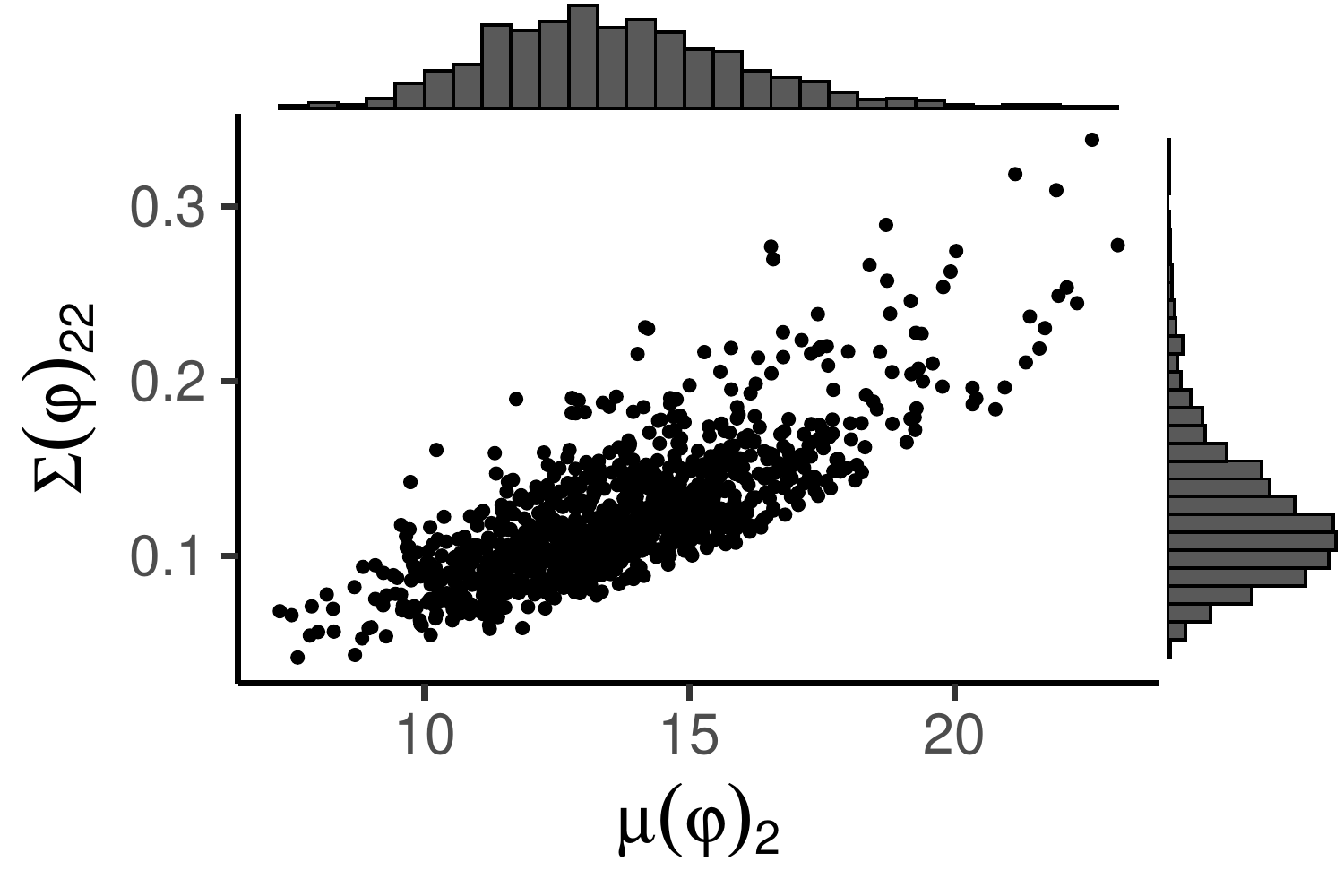} & \includegraphics[width=65mm]{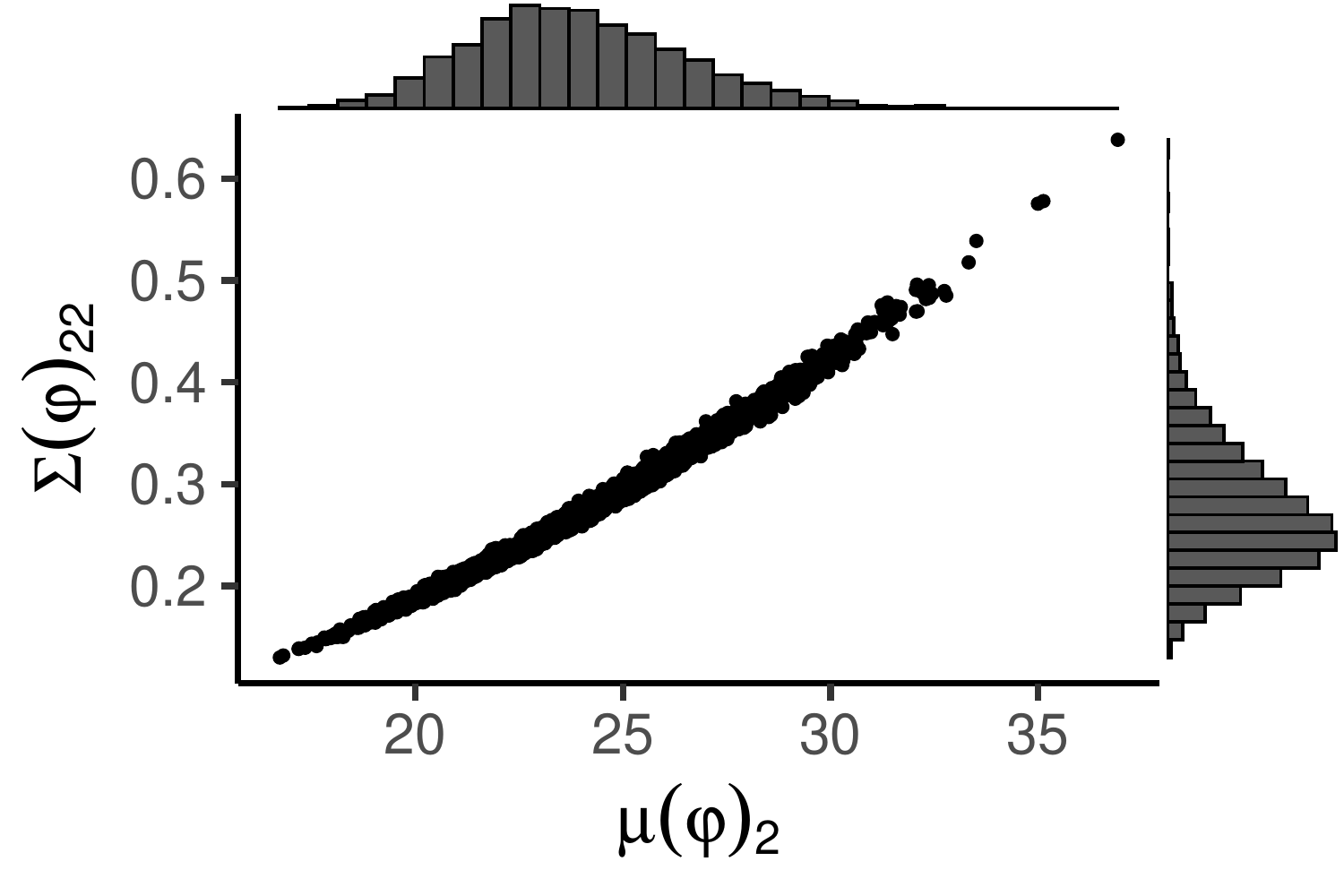}
		\end{tabular}
	\end{center}
\vspace{-.7cm}
	\caption{\label{diagnostics2}Scatterplot of $(\mu(\varphi^{(s)})_2,\Sigma(\varphi^{(s)})_{22})$, $s=1,\dots, S$, 
		for cut posterior (left) and full posterior (right) samples.  Histograms of each variable are shown on the axes.}
\end{figure}
A comparison of the marginal SMI approach of Section 3.4 with the SMI of \cite{carmona2020semi} is given in the supplementary material.  
The two methods give similar results in this example.

\subsection{A random effects model}\label{sec:re}

Our second example, discussed in \cite{liu+bb09}, considers a random
effects model.  The data are denoted by $Y_{ij}$, $i=1,\dots, N$, $j=1,\dots, J$, 
where $i$ indexes groups, and $j$ indexes observations
within groups.  The data for group $i$ is modelled as 
$$Y_{ij}|\beta_i,\varphi_i \stackrel{iid}{\sim} N(\beta_i,\varphi_i^2),\;\;\;j=1,\dots, J,$$
where $\beta_i$ is a random effect, and $\varphi_i$ is a group standard
deviation.  
The prior density for
$\beta$ is
$$\beta_i|\psi \stackrel{iid}{\sim} N(0,\psi^2),$$
$i=1,\dots, N$, where $\psi$ is the random effects standard deviation.  
\cite{liu+bb09} consider this example to demonstrate a problem that
can occur for some hierarchical models, in which there is a 
model for the random effects with thin tails, such as Gaussian.  
In the model above, if there is an outlying value for one
of the random effects, this can lead 
to poor inference for the corresponding group standard deviation, 
and overshrinkage in estimating the random effect.  The difficulty 
is most pronounced  when the number of replicates $J$ is small compared to
$N$.  \cite{liu+bb09} give an insightful discussion
that exploits the simple form of the model to do analytic calculations.  
We do not repeat their analysis here, 
but demonstrate the problem numerically
and illustrate the utility of our generalized Bayes approaches to
modular inference.

First, we will set up the model so that it takes the form of a
two module system.
Write $\beta=(\beta_1,\dots, \beta_N)^\top$ and
$\varphi=(\varphi_1,\dots, \varphi_N)^\top$.  
Let $\eta=(\beta^\top,\psi)^\top$.  
We use similar priors to \cite{liu+bb09}, although we parametrize
our model in terms of standard deviations rather than variances
and transform priors appropriately.  
Components of $\varphi$ are independent in the prior, with
marginal densities 
$\pi(\varphi_i)\propto \varphi_i^{-1}$.  For the prior on $\psi$, we use 
$\pi(\psi |\varphi_i)\propto(\bar{\varphi}^2/J+\psi^2)^{-1}\psi$, 
where $\bar{\varphi}^2=N^{-1}\sum_{i=1}^N \varphi_i^2$.  

We will reduce the full data down to sufficient statistics.  
Let $w_i=J^{-1}\sum_{j=1}^J Y_{ij}$, $z_i=\sum_{j=1}^J (Y_{ij}-z_i)^2$, 
$i=1,\dots, N$, and write $\bz=(z_1,\dots, z_n)^\top$, $\bw=(w_1,\dots, w_n)^\top$.  It is easily seen that $z$ and $\bw$ are sufficient for $\theta=(\varphi^\top,\eta^\top)^\top$, with
$\bz$ and $\bw$ being independent of each other.  The density of $\bz|\varphi$, 
written $p(\bz|\varphi)$, depends only on $\varphi$, with
$$z_i|\varphi_i\sim \text{Gamma}\left(\frac{J-1}{2},\frac{1}{2\varphi_i^2}\right),$$
independently for $i=1,\dots, N$.  
Similarly, write $p(\bw|\varphi,\eta)$ for the density of $w$, and  
$$w_i|\beta_i,\varphi_i\sim N\left(\beta_i,\frac{\varphi_i^2}{J}\right),$$
independently, for $i=1,\dots, N$.  
The model for the sufficient statistics is a two-module system.  The first module consists of $p(\bz|\varphi)$ and $p(\varphi)$, and
the second module comprises $p(\bw|\varphi,\eta)$ and $p(\eta|\varphi)$.  

We simulate a dataset from the model, with $N=100$, $J=10$, $\psi=1$ and
$\varphi_i=0.5$, $i=1,\dots, N$.   The random effects vector $\beta$ is 
simulated from its prior, 
except for $\beta_1$, which is fixed at $10$.  Since $\beta_1$ is inconsistent
with the hierarchical prior, this leads to poor estimation of $\varphi_1$ 
when $J$ is small compared to $N$, and poor estimation of $\beta_1$.  
Figure \ref{RE-example} (left)
compares the posterior distributions of $\varphi_1$ from the conventional
parametric and the cut posterior distributions.   The boxplots 
are for 1,000 posterior samples in each case.  The horizontal line
shows the true value.  The accuracy of the conventional posterior
is poor, and inconsistent with the cut posterior inferences which
are more accurate.

\begin{figure}[h]
	\begin{center}
		\begin{tabular}{c}
			\includegraphics[width=100mm]{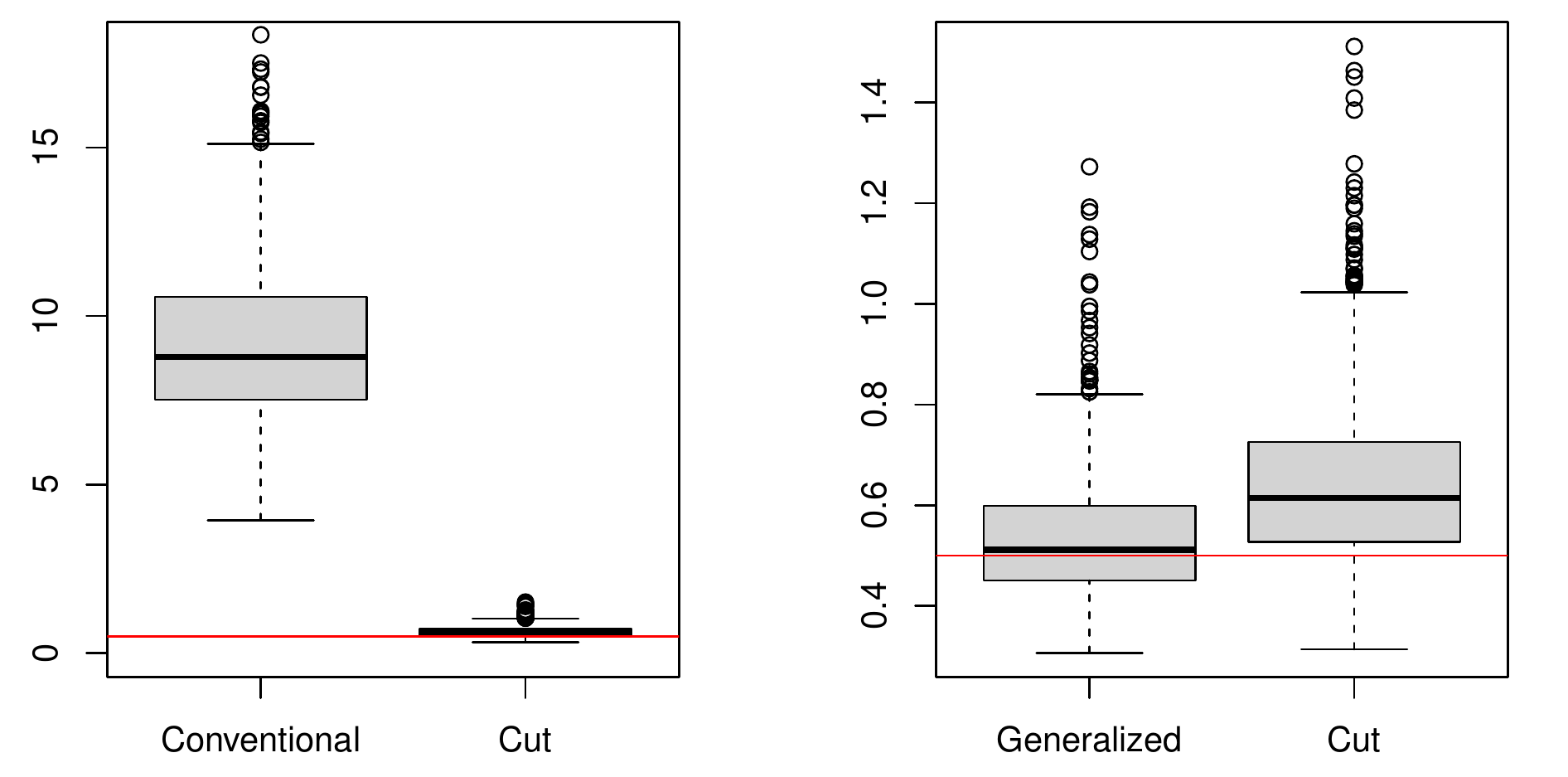} 
		\end{tabular}
	\end{center}\vspace{-.7cm}

	\caption{\label{RE-example}Left: comparison of posterior samples
		for $\varphi_1$ for conventional posterior and cut posterior for
		parametric model specification.  Right:  comparison of posterior 
		samples for $\varphi_1$ for generalized Bayes analysis with
		Tukey's loss for the second module, $\kappa=5$ and $\nu'=3.3$
		with cut posterior.  All boxplots summarize 1,000 posterior samples, and
		the horizontal line is the true value $\varphi_1=0.5$}
\end{figure}

\subsubsection{Generalized posterior analysis}

It is interesting in this example to replace the normal model for $w_i$ 
in module 2 with a loss likelihood, to see whether this resolves the inconsistency
between the cut and full generalized posterior inferences.
Here we consider Tukey's loss \citep{beaton+t74}, which was recently used for a generalized Bayesian
analysis by \cite{jewson+r22}.   
As pointed out
by \cite{jewson+r22}, Tukey's loss can be useful when an analyst 
knows the distribution of the data has heavy tails, but a precise knowledge
of the tail behaviour is difficult to formalize.  
Writing $w_i'=w_i'(\varphi_i,\beta_i)=(w_i-\beta_i)/(\phi_i/\sqrt{J})$, 
in our generalized Bayesian analysis we replace the 
Gaussian negative log-likelihood terms
$$-\log p(w_i|\varphi_i,\beta_i)=\frac{1}{2}\log \frac{2\pi\varphi_i^2}{J}-\frac{1}{2}w_i'^2,$$
with Tukey's loss terms
\begin{align*}
	m(w_i;\eta,\varphi)=\left\{\begin{array}{ll}
		\frac{1}{2}\log \frac{2\pi\varphi_i^2}{J}+\frac{w_i'^2}{2}-\frac{w_i'^4}{2\kappa^2}-\frac{w_i'^6}{6\kappa^4} & \mbox{if $|w_i'|\leq \kappa$} \\
		\frac{1}{2}\log \frac{2\pi\varphi_i^2}{J}+\frac{\kappa^2}{6},
	\end{array}\right.
\end{align*}
for $i=1,\dots, N$, where $\kappa$ is a tuning parameter  
controlling the degree of
robustness to departures from normality.  As $\kappa\rightarrow\infty$, 
Tukey's loss approaches the Gaussian negative log-likelihood, 
whereas small values for $\kappa$ result in greater robustness to outliers.
There are a variety of ways to choose $\kappa$, but here we fix
$\kappa=5$.  \cite{jewson+r22} describe a way of choosing
$\kappa$ and other loss parameters using a so-called ${\cal H}$-posterior
based on the Hyv\"{a}rinen score, and also consider model choice
for loss functions, but these directions are not pursued here.  
For Tukey's loss, the corresponding loss likelihood is not integrable in
$w$, so it does not correspond to any probabilistic model.  

Our generalized Bayesian analysis requires a choice of the learning
rates $\nu$ and $\nu'$ as discussed in Section 3.2. 
Recall that $\nu$ calibrates the module 1 loss to the prior, and
$\nu'$ can be thought of as calibrating the module 2 loss to the conditional
prior for $\eta|\varphi$.  Since we use the original probabilistic specification for module 1, we choose the learning rate $\nu$ to be $1$, and 
the generalized Bayes and conventional cut posterior densities 
for $\varphi$ are the same.  To choose $\nu'$, we use the method
discussed in Section 3.2.  However, noting that only the parameters
$\beta$ appear in the loss function and not the prior hyperparameter 
$\psi$, we calibrate $\nu'$ by considering matching the Fisher information
number for updates for $\beta$ asymptotically with $\psi$ fixed, for 
loss likelihood bootstrap and generalized Bayes.  
Since the matching is done asymptotically, the choice of $\psi$ 
makes no difference to the value of $\nu'$ obtained.  To estimate the matrix $\Psi_{22}$ in estimating $\nu'$ in Section 3.2, we used a Bayesian bootstrap
applied to the original data groups, since it is not possible otherwise
to estimate $\Psi_{22}$ from the data sufficient statistics.  This is because
there is no replication that can be used, 
with $\beta_i$ appearing only in the model for $w_i$.  
The learning rate obtained for the second module for the analysis was
$\nu'=3.3$.   

Figure \ref{RE-example} (right)
compares the posterior distributions of $\varphi_1$ for the generalized
Bayes posterior and the cut posterior distributions.   Once again, the boxplots 
are for 1,000 posterior samples, and the horizontal line shows
the true value.  The cut posterior is the same as for the conventional
posterior for the parametric model, as we are still using the negative log-likelihood as the loss for module 1.  We see that now the cut and full
posterior inferences are consistent with each other, so that 
the Tukey's loss, which accommodates heavy-tailed data, resolves
the conflict between different parts of the model.  
Although we have reduced the full data
to sufficient statistics for inference, the non-sufficient information in the
replicates is useful for model checking - using the replication we
may distinguish between model failure due to outliers in the sampling
density and model failure due to an inappropriate prior on the random
effects.  An outlying random effect for a group will influence all replicates 
in the group.

{ For computations in this example, we used the {\texttt{rstan}} package \citep{carpenter_stan_2017} for both the conventional and generalized
posterior densities.    We ran four chains with 1000 iterations burn-in and
4000 sampling iterations, thinning the output so that 1000 samples are
retained.  The cut posterior density for $\varphi_1^2$ is inverse gamma, and
was sampled directly to get $1,000$ cut posterior samples for  $\varphi_1$.}  

\section{Discussion}

This paper combines generalized posterior inference with cutting feedback
methods for flexible Bayesian modular inference.  
Starting out with a parametric model specification, we can replace 
the negative log likelihood for unreliable modules with different 
choices of a loss function
to resolve any incompatibility between cut and full posterior inferences.  We have also studied the large sample behaviour of 
the generalized cut posterior distribution, taking a conditional 
perspective.  Our main result describes the asymptotic behaviour of
the conditional
posterior distribution of a module's parameters given parameters in other
modules, formally justifying conditional Laplace approximations.  
These provide 
more accurate approximations of conditional posterior distributions than
those obtained from Laplace approximations
of the joint posterior density.  
Our large sample approximations are useful 
for computing diagnostics describing uncertainty
propagation between modules, as well as for the efficient implementation
of a new approach to semi-modular inference.  

In the framework for modular inference that we have
developed, the loss function is a sum of loss functions
associated with different modules.  We considered calibrating
the different component loss functions in one example, but more research
is needed on the best way to do this for different purposes.  With a single loss function, there are different methods of calibrating
the loss to the prior, and the best method to use may depend on the goals
of the analysis.   A similar remark applies in generalized Bayesian modular
inference.  An anonymous referee has also asked about the connections with
the ``restricted likelihood" approach to dealing with misspecification, 
discussed recently in \cite{lewis2021bayesian}.  Restricted likelihood reduces
the data to an insufficient summary statistic, to discard information that
cannot be matched under the assumed model.  The method can be implemented computationally using
likelihood-free inference algorithms, 
and modular inference has been
considered in this context by \cite{chakraborty2023modularized}.  

\bibliographystyle{ba}
\bibliography{mod_bib}
 
\begin{acks}[Acknowledgments]
David Nott is affiliated with the Institute of Operations Research and Analytics at the National University of Singapore.  David Frazier was supported by the Australian Research Council's Discovery Early Career Researcher Award funding scheme (DE200101070). We thank the editorial team for their help in greatly improving the manuscript. 
\end{acks}

\appendix
\section{Proofs of Main Results}

\begin{proof}[Proof of Lemma 1]
	We prove the result by verifying the sufficient conditions in Theorem 1 of \cite{chernozhukov2003mcmc} for the criteria $\nu L_n(\varphi)$. Assumption 1 satisfies the sufficient conditions in Lemmas 1 and 2 in \cite{chernozhukov2003mcmc}, which together with the prior condition in Assumption 1, verifies the sufficient conditions in Theorem 1 of \cite{chernozhukov2003mcmc}. The result follows. 
\end{proof}

\begin{remark}
	The proof of Theorem \ref{prop:bvm1} follows by generalising the arguments in \cite{chernozhukov2003mcmc}. This is a novel generalization for at least two reasons. Firstly, the separability of the criterion functions allow us to maintain different conditions for each portion, e.g., different levels of smoothness, without requiring conditions on the joint criterion, $Q_n(\theta)$. Secondly, by focusing on $\varphi\in\Phi_\delta$, the resulting posterior concentration is not directly impacted by the rate at which the posterior for $\pi_\cut(\varphi|\bz)$ concentrates onto $\varphi^\star$. That is, the result of Theorem \ref{prop:bvm1} remains valid when the posterior for $\varphi$ concentrates at rates slower than the parametric $\sqrt{n}$-rate, so long as Assumption \ref{ass:expand} remains valid. %However, since the parameter space $\Phi$ is assumed to be Euclidean, this result is not applicable to cases where $\Phi$ is a function space. In this way, these results apply in cases where $\varphi$ is parametric, but may concentrate at rates slower than $\sqrt{n}$. 
\end{remark}

\begin{proof}[Proof of Theorem 1]
	To simplify the proof of the result, let us abuse notation and write $n=n_2$. Define $J(\varphi):=J(\eta^\star_\varphi|\varphi)$, and
	$$
	t=\sqrt{n}(\eta-\eta^\star_\varphi)-Z_n(\varphi)/\sqrt{n},\text{ where }Z_n(\varphi):=J(\varphi)^{-1}\Delta_{n}(\eta^\star_\varphi|\varphi).
	$$ From the quadratic approximation in Assumption \ref{ass:expand}, and the above definitions, we have the identity
	\begin{flalign}\label{eq:rewrite}
		M_n(\eta|\varphi)-M_n(\eta^\star_\varphi|\varphi)=-\frac{1}{2}t^\top J(\varphi)t+\frac{1}{2}\frac{1}{\sqrt{n}}Z_{n}(\varphi)^{\top}J^{}(\varphi)\frac{1}{\sqrt{n}}Z_{n}(\varphi)+R_n(\eta,\varphi),
	\end{flalign}for some remainder term $R_n(\eta,\varphi)$. 
	Now, define $T_n(\varphi):=Z_n(\varphi)/n+\eta^\star_\varphi$, and let 
	\begin{equation*}
		\omega(t):=M_n\{T_n(\varphi)+t/\sqrt{n}|\varphi\}-M_n(\eta^\star_\varphi|\varphi)-\frac{1}{2}\frac{1}{\sqrt{n}}Z_{n}(\varphi)^{\top}J^{}(\varphi)\frac{1}{\sqrt{n}}Z_{n}(\varphi),
	\end{equation*}
	which, by \eqref{eq:rewrite}, is equal to
	\begin{equation}\label{eq:omega}
		\omega(t_{})=-\frac{1}{2}t^\top J(\varphi)t+R_n(T_n(\varphi)+t/\sqrt{n},\varphi).
	\end{equation}Using \eqref{eq:omega}, the posterior can be stated as 
	\begin{flalign*}
		\pi(t|\bw,\varphi)&:=\frac{\pi\{t/\sqrt{n}+T_{n}(\varphi)|\varphi\}\exp[\nu\{M_n\{T_n(\varphi)+t/{\sqrt{n}}|\varphi\}-M_n\{\eta^\star_\varphi|\varphi\}\}]}{\int_{\mathcal{E}_n} {\pi\{t/\sqrt{n}+T_{n}(\varphi)|\varphi\}\exp[\nu\{M_n\{T_n(\varphi)+{t}/{\sqrt{n}}|\varphi\}-M_n\{\eta^\star_\varphi|\varphi\}\}]\dt t}}%\\&=\frac{\pi\{t/\sqrt{n}+T_{n}(\varphi)|\varphi\}\exp\{\nu\omega(t)\}}{\int_{\mathcal{E}_n} {\pi\{t/\sqrt{n}+T_{n}(\varphi)|\varphi\}\exp\{\nu\omega(t)\}\dt t}}
		\\&={\pi\{t/\sqrt{n}+T_{n}(\varphi)|\varphi\}\exp\{\nu\omega(t)\}}/C_n,
	\end{flalign*}for 
	$$
	C_n:={\int_{\mathcal{E}_n} {\pi\{t/\sqrt{n}+T_{n}(\varphi)|\varphi\}\exp\{\nu\omega(t)\}\dt t}}.
	$$
	
	The stated result follows if
	\begin{flalign*}
		\int_{\mathcal{E}_n}  \|t\|\left|\pi(t|\bw,\varphi)-N\{t;0,[\nu J(\varphi)]^{-1}\}\right|\dt t&=C_n^{-1}J_n=o_p(1), 
	\end{flalign*}
	where
	\begin{flalign*}
		J_{n}&=\int_{\mathcal{E}_n} \|t\|^{}\bigg{|}\exp\left\{\nu\omega(t)\right\} \pi_{}\left\{T_n(\varphi)+{t_{}}/{\sqrt{n}}|\varphi\right\}-C_nN\{t;0,[\nu J(\varphi)]^{-1}\}\bigg{|} \dt t.
	\end{flalign*}
	However, $J_{n}\leq J_{1n}+J_{2n},$ where
	\begin{flalign*}
		J_{1n}&:= \int_{\mathcal{E}_n} \|t\|\bigg|\exp \left\{ -\frac{1}{2}t^\top [\nu J(\varphi)]t+\nu R_n(T_n(\varphi)+t/\sqrt{n},\varphi)\right\} \pi_{}\left\{T_n(\varphi)+{t}/{\sqrt{n}}|\varphi\right\}\\&\quad-\pi(\eta^\star_\varphi|\varphi)\exp \left\{-\frac{1}{2} t^{\top} [\nu J(\varphi)]^{}t\right\} \bigg| \dt t,\\
		J_{2n}&:=\left|C_{n}\frac{\left|\nu J(\varphi)\right|^{1/2 }}{(2\pi)^{{d_\eta/2}}}-\pi(\eta_\varphi^\star|\varphi)\right|\int_{\mathcal{E}_n}  \|t\|^{}\exp \left\{-\frac{1}{2} t^{\top} [\nu J^{}(\varphi)] t\right\} \dt t .
	\end{flalign*}and where we have used equation \eqref{eq:omega}  in the definition of $J_{1n}$. 
	Further, if $J_{1n}=o_{p}(1)$, then 
	\begin{flalign*}
		C_n&=\pi(\eta^\star_\varphi|\varphi)\int_{\mathbb{R}^{d_\eta}}\exp\left\{-\frac{1}{2}t^{\top}[\nu J(\varphi)] t\right\}\dt t+o_p(1)=\pi(\eta^\star_\varphi|\varphi)\frac{(2\pi)^{d_\eta/2}}{|\nu J(\varphi)|^{1/2}}+o_p(1),
	\end{flalign*}and $J_{2n}=o_p(1)$ since for each $\varphi\in\Phi_\delta$ the matrix $J(\varphi)$ is positive-definite, by Assumption 4(ii), so that $\int_{\mathbb{R}^{d_\eta}} \|t\|^{}\exp \left\{-\frac{1}{2} t^{\top} [\nu J^{}(\varphi)] t\right\}\dt t<\infty$.

	Consequently, the result follows if $J_{1n}=o_p(1)$. Inspecting  $J_{1n}$ it is clear that the specific value of $\nu$ will  not impact whether or not $J_{1n}=o_p(1)$, so long as $\nu>0$. Since $\nu$ is fixed it is without loss of generality to take $\nu=1$ in the remainder.

	To demonstrate that $J_{1n}=o_p(1)$ we split $\mathcal{E}_n$ into three regions and analyze $J_{1n}$ over each region. For some $0\le h<\infty$ and $\gamma>0$, with $\gamma=o(1)$, the regions are defined as follows:
	\begin{itemize}
		\item \textbf{Region 1}: $ \|t\|\leq h$.
		\item \textbf{Region 2}: $  h<\|t\|\leq \gamma \sqrt{n}$.
		\item \textbf{Region 3}: $  \|t\|\geq \gamma \sqrt{n}$.
	\end{itemize}

	The remainder of the proof follows by extending similar arguments in the literature, e.g., Theorem 8.2 in \cite{lehmann2006theory} (pg 489), and Theorem 1 in  \cite{chernozhukov2003mcmc}, to accommodate the conditional nature of the result.

	\medskip

	\noindent\textbf{\textbf{Region 1}:} Over this region $\|t\|$ can be neglected and the result follows if
	$$
	\sup_{\|t\|\le h}\left|\exp\left\{\omega(t)\right\} \pi_{}\{T_n(\varphi)+{t}/{\sqrt{n}}|\varphi\}-\pi(\eta^\star_\varphi|\varphi)\exp \{- t^{\top} J^{}(\varphi) t/2\}\right|=o_p(1).
	$$Now, 
	\begin{flalign*}
		&\left|\exp\left\{\omega(t)\right\} \pi_{}\{T_n(\varphi)+{t}/{\sqrt{n}}|\varphi\}-\pi(\eta^\star_\varphi|\varphi)\exp \left\{-t^{\top} J^{}(\varphi) t/2\right\}\right|\\&\le\exp\{\omega(t)\}|\pi_{}\{T_n(\varphi)+{t}/{\sqrt{n}}|\varphi\}-\pi(\eta^\star_\varphi|\varphi)|+\pi(\eta^\star_\varphi|\varphi)|\exp\{\omega(t)\}-\exp \{-t^{\top} J^{}(\varphi)t/2\}|.
	\end{flalign*}First, note that by Assumption 4(i-ii), $Z_n(\varphi)/\sqrt{n}=O_p(1)$ for each $\varphi\in\Phi_\delta$; hence, from the definition $T_n(\varphi)=Z_n(\varphi)/n+\eta^\star_\varphi$, for each $\varphi\in\Phi_\delta$,
	\begin{equation}\label{eq:x1}
		\sup_{\|t\|\le h}|T_n(\varphi)+t_{}/\sqrt{n}-\eta^\star_\varphi|=O_p(1/\sqrt{n}).
	\end{equation}
	
	From Assumption 3, $\pi(\cdot|\varphi)$ is continuous in the first argument, so that by \eqref{eq:x1}, 
	$$
	\sup_{\|t\|\le h}|\pi_{}\{T_n(\varphi)+{t}/{\sqrt{n}}|\varphi\}-\pi(\eta^\star_\varphi|\varphi)|=o_p(1).
	$$
	Also, from \eqref{eq:x1} and Assumption 4(iii), for each $\varphi\in\Phi_\delta$, $$\sup_{\|t\|\le h}|R_n\{T_n(\varphi)+t/\sqrt{n},\varphi\}|=o_p(1);$$ using the equation for  $\omega(t)$ in \eqref{eq:omega}, we then have 
	$$
	\sup_{\|t\|\le h}|\exp\{\omega(t)\}-\exp \{-t^{\top} J^{}(\varphi)t/2\}|=o_p(1).
	$$
	
	Further, since $\sup_{\|t\|\le h}|R_n\{T_n(\varphi)+t/\sqrt{n},\varphi\}|=o_p(1),$ we have $\exp\{\omega(t)\}\le\{1+o_p(1)\}$ over $\|t\|\le h$; since $\pi(\eta|\varphi)$ is continuous in $\eta$ for all $\varphi\in\Phi_\delta$, it follows that $\pi(\eta|\varphi)$ is bounded for $\eta\in\{\|\eta-\eta^\star_\varphi\|\le h/\sqrt{n}\}$. Hence, $J_{1n}=o_{p}(1)$ over Region 1.
	
	\bigskip

	\noindent\textbf{\textbf{Region 2}:}
	For $h$ large enough and $\gamma=o(1)$,  ${J}_{1n}\leq C_{1n}+C_{2n}+C_{3n}$ where
	\begin{flalign*}
		C_{1n}:=&C\sup _{h\le\|t\| \leq \gamma \sqrt{n}}\exp\left\{|R_n\{T_n(\varphi)+t/\sqrt{n},\varphi\}|\right\}\left| \pi_{}\{T_n(\varphi)+{t}/{\sqrt{n}}|\varphi\}-\pi_{}\left(\eta^\star_\varphi|\varphi\right)\right|\\&\times\int_{h\leq \|t\| \leq \gamma \sqrt{n}}\|t\|\exp\{-t^{\top}J^{}(\varphi)t/2\}\dt t \\
		C_{2n}:=&C\int_{h\leq \|t\| \leq \gamma \sqrt{n}}\|t\| \exp\{-t^{\top}J^{}(\varphi)t/2\}\exp\left\{|R_n\{T_n(\varphi)+t/\sqrt{n},\varphi\}|\right\} \pi_{}\{T_n(\varphi)+{t}/{\sqrt{n}}|\varphi\} \dt t \\C_{3n}:=&C \int_{h\leq \|t\| \leq \gamma \sqrt{n}}\|t\|\exp\{-t^{\top}J^{}(\varphi)t/2\}\dt t .
	\end{flalign*}
	
	The first term satisfies $C_{1n}=o_{p}(1)$ for any fixed $h$, so that $C_{1n}=o_{p}(1)$ for $h\rightarrow\infty$, by the dominated convergence theorem. For  $C_{3n}$, from the continuity and positive definiteness of $J(\varphi)$, for each $\varphi\in\Phi_\delta$, there exists $h'$ large enough such that for all $h>h'$, and $\|t\|\ge h$ $$\|t\|\exp\left\{-t^{\top}J_{}(\varphi)t/2\right\}\le \|t\|\exp[-\|t\|^2\lambda_{\text{min}}\{J(\varphi)\}]=O(1/h),$$where $\lambda_{\text{min}}(M)$ denotes the minimum eigenvalue of the matrix $M$. Hence, $C_{3n}$ can be made arbitrarily small by taking $h$ large enough and $\gamma$ small enough.
	
	To demonstrate that $C_{2n}=o_p(1)$, we show that 
	\begin{equation}\label{eq:bound1}
		\exp\{-t^{\top}J(\varphi) t/2\}\exp\left\{|R_n\{T_n(\varphi)+t/\sqrt{n},\varphi\}|\right\}\pi\{T_n(\varphi)+t/\sqrt{n}|\varphi\}\le C\exp\left\{-t^{\top}J(\varphi) t/4\right\},
	\end{equation}with probability converging to one (wpc1), so that $C_{2n}$ can be bounded above by
	\begin{flalign*}
		C_{2n}\le &C\int_{h\le\|t\|\le\gamma \sqrt{n}}\|t\|\exp\left\{-t^{\top}J(\varphi) t/4\right\}\dt t.
	\end{flalign*}Similar to $C_{1n}$ and $C_{3n}$, the RHS of the above can be made arbitrarily small for some $h$ large and $\gamma$ small. 
	
	To demonstrate equation \eqref{eq:bound1}, first note that by continuity of $\pi(\eta|\varphi)$, Assumption 3, $\pi\{T_n(\varphi)+t/\sqrt{n}|\varphi\}$ is bounded over $\{t_{}:h\le \|t_{}\|\le\gamma \sqrt{n}\}$ for each $\varphi\in\Phi_\delta$ and can be dropped from the analysis. Now, since $\|T_n(\varphi)-\eta^\star_\varphi\|=o_p(1)$, for any $\gamma>0$, $\|T_n(\varphi)+t/\sqrt{n}-\eta^\star_\varphi\|<2\gamma$ for all $\|t\|\le\gamma \sqrt{n}$ and $n$ large enough. Therefore, by Assumption 4(iii), there exists some $\gamma'>0$ and $h$ large enough so that
	\begin{align*}
		\sup_{h\le\|t\|\le\gamma' \sqrt{n}}|R_n\{T_n(\varphi)+t/\sqrt{n},\varphi\}| \le & \frac{1}{4}\lambda_{\text{min}}\{J^{}(\varphi)\}\{1+\|t+Z_n(\varphi)/\sqrt{n}\|^2\} \\
		\leq & \frac{1}{4}\lambda_{\text{min}}\{J^{}(\varphi)\}\|t\|^2+O_p(1),
	\end{align*}
	where the last inequality follows since $\|Z_n(\varphi)/\sqrt{n}\|=O_p(1)$, for each $\varphi\in\Phi_\delta$ by Assumption 4(i). Thus,  for some $C>0$, wpc1,
	\begin{flalign*}
		\exp\{\omega(t)\}&\leq \exp\left\{-\frac{1}{2}t^{\top}J(\varphi)t+|R_n\{T_n(\varphi)+t/\sqrt{n},\varphi\}|\right\}\le C\exp\left\{-t^{\top}J^{}(\varphi)t/4\right\}. 
	\end{flalign*}Since the result holds for arbitrary and fixed $\varphi$, it holds for each $\varphi\in\Phi_\delta$.

	\medskip

	\noindent\textbf{\textbf{Region 3}:} For $\gamma \sqrt{n}$ large,
	$
	\int_{\|t\|\ge \gamma \sqrt{n}}\|t\|N\{t;0,J^{-1}(\varphi)\}\dt t,$ can be made arbitrarily small and is therefore dropped from the analysis. Using the definition of $\omega(t)$, and the identity $\eta=Z_n(\varphi)/n+t/\sqrt{n}-\eta^\star_\varphi$, consider
	\begin{flalign*}
		{J}_{1n} := & \int_{\|t\|\ge \gamma \sqrt{n}}\|t\| \exp\{\omega(t) \} \pi\{Z_n(\varphi)/n+t / \sqrt{n}-\eta^\star_\varphi|\varphi\}\dt t \\
		= &\sqrt{n}^{d_\eta+1}\int_{\|\eta-Z_n(\varphi)/n-\eta^\star_\varphi\|\ge \gamma }\|\eta-Z_n(\varphi)/n-\eta^\star_\varphi\|^{} \times \\
		& \exp\left\{M_n(\eta|\varphi)-M_n(\eta^\star_\varphi|\varphi)-\frac{1}{2n}Z_n(\varphi)^{\top} J(\varphi)^{-1}Z_n(\varphi)\right\} \pi\left(\eta|\varphi\right)\dt \eta\\ \le & O_p(1)\sqrt{n}_{n}^{d_\eta+1}\int_{\|\eta-Z_n(\varphi)/n-\eta^\star_\varphi\|\ge \gamma }\|\eta-\eta_\varphi^\star\|\exp\left\{M_n(\eta|\varphi)-M_n(\eta^\star_\varphi|\varphi)\right\} \pi\left(\eta|\varphi\right)\dt \eta,
	\end{flalign*}
	since $n^{-1}Z_n(\varphi)^{\top} J(\varphi)^{-1}Z_n(\varphi)=O_{p}(1)$ under Assumption 4(i) for each $\varphi\in\Phi_\delta$.
	
	From Assumption 2(ii), for fixed $\delta_1>0$, and any $\delta_2>0$, there exists an $\epsilon=\epsilon(\delta_1,\delta_2)>0$ such that
	$$
	\sup_{\varphi\in\Phi_{\delta_1}}\sup_{\|\eta-\eta^\star_\varphi\|\ge \delta_2}\left\{\mathbb{M}(\eta|\varphi)-\mathbb{M}(\eta^\star_\varphi|\varphi)\right\}\le -\epsilon.
	$$Therefore, the above and the uniform convergence in Assumption 2(i) together imply that
	\begin{equation}
		\label{eq:expconv}
		\lim_{n\rightarrow\infty}P^{(n)}_0\left[\sup_{\varphi\in\Phi_{\delta_1}}\sup_{\|\eta-\eta^\star_\varphi\|\geq \delta_2}\exp\left\{M_n(\eta|\varphi)-M_n(\eta^\star_\varphi|\varphi)\right\}\leq \exp(-\epsilon n^2)\right]=1.
	\end{equation}
	Since for each $\varphi\in\Phi_\delta$, $Z_n(\varphi)/\sqrt{n}=O_p(1)$, by Assumption 4(i), from equation \eqref{eq:expconv} we obtain
	\begin{align*}
		{J}_{1n} & \le \{1+o_p(1)\}O_p(1)\sqrt{n}^{d_\eta+1}\int_{\|\eta-\eta^\star_\varphi\|\ge \gamma }\|\eta-\eta^\star_\varphi\|^{}\pi\left(\eta|\varphi\right)\exp\{M_n(\eta|\varphi)-M_n(\eta^\star_\varphi|\varphi)\}\dt \eta\\&\leq O_p(1) \exp\left(-\epsilon n^2\right)\sqrt{n}^{d_\eta+1}\int_{\|\eta-\eta^\star_\varphi\|\ge \gamma }\|\eta-\eta^\star_\varphi\|^{}\pi\left(\eta|\varphi\right)\dt\eta\\&\le O_p(1) \exp\left(-\epsilon n^2\right)\sqrt{n}^{d_\eta+1}\left\{\int_{\mathcal{E}}\|\eta\|^{}\pi\left(\eta|\varphi\right)\dt\eta+ \|\eta^\star_\varphi\|\right\}.
	\end{align*}By Assumption 3(ii),  $\int_{\mathcal{E}}\|\eta\|\pi(\eta|\varphi)\dt\eta<\infty$ for each  $\varphi\in\Phi_{\delta_1}$. By Assumptions 1-2, $\eta^\star_\varphi$ exists  for each $\varphi\in\Phi_{\delta_1}$. Hence, 
	$$
	J_{1n}\le O_p(\exp\{-\epsilon n^2\}\sqrt{n}^{d_\eta+1})=o_p(1),
	$$ for each $\varphi\in\Phi_{\delta_1}$, and some $\delta_1>0$.

	Placing all three regions together we obtain
	\begin{equation}
		\int_{\mathcal{E}_n}|\pi(t|\bw,\varphi)-N\{t;0,[\nu J(\eta^\star_\varphi|\varphi)]^{-1}\}|=o_p(1).
		\label{eq:over_res}
	\end{equation}for each $\varphi\in\Phi_\delta$.	The result follows by using the fact that the total variation norm is invariant under a change of location. Namely, the result stated in the theorem follows from \eqref{eq:over_res} by defining 
	$
	\xi=\sqrt{n}(\eta-\eta^\star_\varphi), 
	$ and consider the change in location from $t:=\sqrt{n}(\eta-\eta^\star_\varphi)-Z_n(\varphi)/\sqrt{n}$ to $\xi=t+Z_n(\eta^\star|\varphi)/\sqrt{n}=\sqrt{n}(\eta-\eta^\star_\varphi).$

\end{proof}

\section{Joint Behavior of cut posterior}\label{sec:joint}

While we argue that the conditional view of the posterior for $\eta$ presented in Theorem 1 is most appropriate, it is feasible to obtain a large sample result for the joint cut posterior. 
To obtain such a result, we require smoothness conditions, in $\varphi$, for the functions $\Delta_{n_2}(\eta|\varphi)$ and $J(\eta|\varphi)$ in Assumption 4. Further, we assume $\Delta_{n_2}(\eta|\varphi)$ is differentiable in $\varphi$, but this can be weakened to stochastic differentiability at the cost of additional technicalities.

Throughout the remainder of this section, to make clear that we are considering joint inference on $\theta=(\varphi^{\top},\eta^{\top})^{\top}$, rather than conditional inference for $\eta\mid\varphi$, we abuse notation and write terms that depend on both $\eta,\varphi$ as $(\eta,\varphi)$ and not $\eta\mid\varphi$; e.g., we write $\Delta_{n_2}(\eta,\varphi)$ and $J(\eta,\varphi)$, rather than using the conditioning notation.

\begin{assumption}\label{ass:fur_exp}
	For $\varphi\in\Phi_\delta$, and $\Delta_{n_2}(\eta,\varphi)$, $J(\eta,\varphi)$ as in Assumption \ref{ass:expand}, the following are satisfied: (i) $\nabla^2_{\eta\varphi}\mathbb{M}(\eta^\star,\varphi)$ and $J(\eta,\varphi)$ are both continuous in $\varphi$;  (ii) $\nabla_\varphi\Delta_{n_2}(\eta^\star,\varphi)$ exists and satisfies $\sup_{\varphi\in\Phi_\delta}\|\frac{1}{n_2}\nabla_\varphi\Delta_{n_2}(\eta^\star,\varphi)-\nabla_{\eta\varphi}\mathbb{M}(\eta^\star,\varphi)\|=o_p(1)$. 
\end{assumption}

To present the joint distribution of the cut posterior, we require a few additional notations. Define
$$
\Sigma:= \begin{pmatrix}
	\Sigma_{11}& \Sigma_{12}\\
	\Sigma_{21}&\Sigma_{22}
\end{pmatrix}=\nu\begin{pmatrix}
	-\nabla^2_{\varphi\varphi}\mathbb{L}(\varphi^\star)& \nabla^2_{\varphi\eta}\mathbb{M}(\eta^\star,\varphi^\star)\\
	\nabla^2_{\eta\varphi}\mathbb{M}(\eta^\star,\varphi^\star)&-\nabla^2_{\eta\eta}\mathbb{M}(\eta^\star,\varphi^\star)
\end{pmatrix}
$$ and recall that $\zeta=\lim_{n_1,n_2\rightarrow\infty}n_1/n_2$, with $0<\zeta<\infty$, and let $\vartheta:=\zeta^{-1/2}$.  Define
$$	V:=\begin{pmatrix}
	V_{11}&V_{12}\\V_{21}&V_{22}
\end{pmatrix}
=
\begin{pmatrix}
	\Sigma_{11}^{-1}&-\vartheta\cdot \Sigma_{11}^{-1}\Sigma_{12}\Sigma_{22}^{-1}\\-\vartheta\cdot \Sigma_{22}^{-1}\Sigma_{21}\Sigma_{11}^{-1}&\Sigma_{22}^{-1}+\vartheta^{2}\Sigma_{22}^{-1}\Sigma_{21}\Sigma_{11}^{-1}\Sigma_{12}\Sigma_{22}^{-1}
\end{pmatrix},
$$and note that by block matrix inversion we have
\begin{equation}\label{eq:eq_v} 
	V^{-1}:=\begin{pmatrix}
		\Sigma_{11}+\vartheta^2\cdot\Sigma_{12}\Sigma_{22}^{-1}\Sigma_{21}&\vartheta\cdot\Sigma_{12}\\\vartheta\cdot\Sigma_{21}&\Sigma_{22}
	\end{pmatrix}
\end{equation}
In addition, define $$D_n=\begin{pmatrix}
	n_1\cdot I_{d_\varphi}&0\\0&n_2\cdot I_{d_\eta}
\end{pmatrix}$$ and $Z_n=(Z_{n_1}^{\top},Z_{n_2}^{\top})^{\top}$, where
\begin{align*}
	\begin{pmatrix}
		Z_{n_1}\\Z_{n_2}
	\end{pmatrix} & :=
	\begin{pmatrix}\Sigma_{11}^{-1}&0\\-\vartheta\cdot \Sigma_{22}^{-1}\Sigma_{12}\Sigma_{11}^{-1}&\Sigma_{22}^{-1}		
	\end{pmatrix}D_n^{-1/2}\begin{pmatrix}\nabla_\varphi L_{n_1}(\varphi^\star)\\\Delta_{n_2}(\eta^\star,\varphi^\star)\end{pmatrix} \\
	& =\begin{pmatrix}
		\Sigma_{11}^{-1}\nabla_\varphi L_{n_1}(\varphi^\star)/\sqrt{n_1}\\\Sigma_{22}^{-1}\Delta_{n_2}(\eta^\star,\varphi^\star)/\sqrt{n_2}-\vartheta \cdot\Sigma_{22}^{-1}\Sigma_{12}\Sigma_{11}^{-1}\nabla_\varphi L_{n_1}(\varphi^\star)/\sqrt{n_1}
	\end{pmatrix},
\end{align*}
and define  $$\phi:=\sqrt{n_1}(\varphi-\varphi^\star)-Z_{n_1},\quad\xi:=\sqrt{n_2}(\eta-\eta^\star)-Z_{n_2}, \quad t:=(\phi^{\top},\xi^{\top})^{\top},\quad T_n:=D_n^{-1/2}Z_n+\theta^\star.$$ The cut posterior for $t$ is then given by $\pi_{\cut}(t|\bz,\bw)=|D_n|^{-1/2}\pi_\cut(D_n^{-1/2}t+T_n\mid \bz,\bw)$, which has support $\mathcal{T}_n:=\{t=D_n^{1/2}(\theta-\theta^\star)-Z_n:\theta\in\Theta\}$.  

\begin{corollary}\label{corr:bvm_rev}
	Under Assumptions \ref{ass:ident_L}-\ref{ass:expand} and \ref{ass:fur_exp},
	$
	\int_{\mathcal{T}_n}\left|\pi_{\cut}(t|\bz,\bw)-N\left\{t;0,V\right\}\right|\dt t=o_p(1)
	$.
\end{corollary}

Corollary \ref{corr:bvm_rev} extends the results obtained by \cite{pompe2021asymptotics} to cut posterior densities based on arbitrary criterion functions. Using boundedness and differentiability assumptions, and a Taylor series approximation, \cite{pompe2021asymptotics} derive a Laplace approximation to the cut posterior via an expansion of the log joint cut posterior. Our results extend theirs in several ways: 1) our smoothness conditions imposed on $M_n(\eta,\varphi)$ and $L_n(\varphi)$ are weaker than those used in \cite{pompe2021asymptotics}; and 2) our results are valid for a wide range of criterion functions one may wish to choose, including quasi-likelihoods, tempered likelihoods, or any other M-estimation criterion.

Corollary \ref{corr:bvm_rev} is presented in a slightly different manner from Proposition 3 in \cite{pompe2021asymptotics}. Our result considers the posterior behavior of  $t=D_n^{1/2}(\theta-\theta^\star)-Z_n
$, while the analysis of \cite{pompe2021asymptotics} considers the posterior behavior of $\sqrt{n_2}(\theta-\theta^\star)$. In this way, the scaling constants in Proposition 3 of \cite{pompe2021asymptotics} differ from those in Corollary \ref{corr:bvm_rev}. Since the rates of convergence for the two components, $\varphi$ and $\eta$, are different, we believe it more direct to consider $t$, which cleanly disentangles the two rates, rather than to bundle the two rates together as in the result of \cite{pompe2021asymptotics}. A result for $\sqrt{n_2}(\theta-\theta^\star)$ can be obtained by instead considering the behavior of the random variable $n_2^{-1/2}D_n^{1/2}(\theta-\theta^\star)-Z_n/\sqrt{n_2}$.

The following result follows immediately from Corollary \ref{corr:bvm_rev} using standard arguments (see, e.g., Theorem 8.3 on page 490 of \citealp{lehmann2006theory}), and the proof is therefore omitted for brevity. 
\begin{corollary}\label{corr:mean}
	If $D_n^{-1/2}\begin{pmatrix}\nabla_\varphi L_{n_1}(\varphi^\star)^\top,&\Delta_{n_2}(\eta^\star,\varphi^\star)^\top\end{pmatrix}^\top\Rightarrow N(0,\Omega)$, then for $\bar\theta_n:=\int_\Theta\theta\pi_\cut(\theta|\bz,\bw)\dt\theta$, we have that
	$$
	D_n^{1/2}(\bar\theta_n-\theta^\star)\Rightarrow N\left(0,\begin{pmatrix}\Sigma_{11}^{-1}&0\\-\vartheta\cdot \Sigma_{22}^{-1}\Sigma_{12}\Sigma_{11}^{-1}&\Sigma_{22}^{-1}		
	\end{pmatrix}\Omega\begin{pmatrix}\Sigma_{11}^{-1}&-\vartheta\cdot \Sigma_{11}^{-1}\Sigma_{12}\Sigma_{22}^{-1}\\0&\Sigma_{22}^{-1}		
	\end{pmatrix}\right).
	$$
	
\end{corollary}

Taken together, Corollaries \ref{corr:bvm_rev} and \ref{corr:mean} demonstrate that the cut posterior does not correctly quantify uncertainty for the posterior mean. In the case of a correctly specified likelihood criterion, Corollary \ref{corr:mean} demonstrates that the posterior mean will not have the same asymptotic variance as the maximum likelihood estimator since it neglects the term $\nabla_{\varphi\varphi}^2\mathbb{M}(\eta^\star,\varphi^\star)$. Therefore, the posterior mean of the cut posterior will be inefficient if the model is correctly specified.

\begin{proof}[Proof of Corollary \ref{corr:bvm_rev}]
	To simplify the proof we take $\nu=1$ in what follows.  Use Assumption \ref{ass:fur_exp} to expand $\Delta_{n_2}(\eta^\star,\varphi)/\sqrt{n_2}$ as 
	\begin{flalign*}
		\frac{\Delta_{n_2}(\eta^\star,\varphi)}{\sqrt{n_2}}&=\frac{1}{\sqrt{n_2}}\Delta_{n_2}(\eta^\star,\varphi^\star)+\nabla_{\eta\varphi}^2 \mathbb{M}(\eta^\star,\varphi^\star)\sqrt{n_2}(\varphi-\varphi^\star) \\&+\{\nabla_\varphi\Delta_{n_2}(\eta^\star,\bar\varphi)/n_2-\nabla_{\eta\varphi}^2 \mathbb{M}(\eta^\star,\bar\varphi)\}\sqrt{n_2}(\varphi-\varphi^\star)\\&+\{\nabla_{\eta\varphi}^2 \mathbb{M}(\eta^\star,\bar\varphi)-\nabla_{\eta\varphi}^2 \mathbb{M}(\eta^\star,\varphi^\star)\}\sqrt{n_2}(\varphi-\varphi^\star)\\&=\frac{\Delta_{n_2}(\eta^\star,\varphi^\star)}{\sqrt{n_2}}+\nabla_{\eta\varphi}^2 \mathbb{M}(\eta^\star,\varphi^\star)\sqrt{n_2}(\varphi-\varphi^\star)+o_p(1)\\&+\{\nabla_{\eta\varphi}^2 \mathbb{M}(\eta^\star,\bar\varphi)-\nabla_{\eta\varphi}^2 \mathbb{M}(\eta^\star,\varphi^\star)\}\sqrt{n_2}(\varphi-\varphi^\star)\\&=\frac{\Delta_{n_2}(\eta^\star,\varphi^\star)}{\sqrt{n_2}}+\nabla_{\eta\varphi}^2 \mathbb{M}(\eta^\star,\varphi^\star)\cdot\vartheta\cdot \sqrt{n_1}(\varphi-\varphi^\star)+o_p(1)\\&+\{\nabla_{\eta\varphi}^2 \mathbb{M}(\eta^\star,\bar\varphi)-\nabla_{\eta\varphi}^2 \mathbb{M}(\eta^\star,\varphi^\star)\}\cdot\vartheta\cdot \sqrt{n_1}(\varphi-\varphi^\star),
	\end{flalign*}for some intermediate value satisfying $\|\bar\varphi-\varphi^\star\|\le\|\varphi-\varphi^\star\|$, and where the $o_p(1)$ term follows by applying Assumption \ref{ass:fur_exp}(ii). From equation A.1 in the proof of Theorem 1, we have that, for 
	\begin{flalign}\label{eq:a12}
		M_n(\eta,\varphi)-M_n(\eta^\star,\varphi^\star)=-\frac{1}{2}t_\varphi^{\top}J(\varphi^\star)t_\varphi+\frac{1}{2}Z_{n_2}(\varphi^\star)^{\top}J^{}(\varphi^\star)Z_{n_2}(\varphi^\star)+R_n(\eta,\varphi),
	\end{flalign}for some remainder term $R_n(\eta,\varphi)$. However, using the above expansion we have
	\begin{flalign*}
		t_\varphi&=\sqrt{n_2}(\eta-\eta^\star)-J(\varphi)^{-1}\Delta_n(\eta^\star,\varphi)
		\\&	=\sqrt{n_2}(\eta-\eta^\star)-J(\varphi)^{-1}\left\{\frac{\Delta_{n_2}(\eta^\star,\varphi^\star)}{\sqrt{n_2}}+\nabla_{\eta\varphi}^2 \mathbb{M}(\eta^\star,\varphi^\star)\vartheta\cdot \sqrt{n_1}(\varphi-\varphi^\star)\right\}+o_p(1)\\&+J(\varphi)^{-1}\{\nabla_{\eta\varphi}^2 \mathbb{M}(\eta^\star,\bar\varphi)-\nabla_{\eta\varphi}^2 \mathbb{M}(\eta^\star,\varphi^\star)\}\vartheta\cdot \sqrt{n_1}(\varphi-\varphi^\star).
	\end{flalign*}
	Since $J(\varphi)^{-1}$ is continuous in $\varphi$, we have that 
	\begin{flalign*}
		t_\varphi&=\sqrt{n_2}(\eta-\eta^\star)-J(\varphi)^{-1}\Delta_n(\eta^\star,\varphi)
		\\&	=\sqrt{n_2}(\eta-\eta^\star)-J(\varphi^\star)^{-1}\left\{\frac{\Delta_{n_2}(\eta^\star,\varphi^\star)}{\sqrt{n_2}}+\nabla_{\eta\varphi}^2 \mathbb{M}(\eta^\star,\varphi^\star)\cdot\vartheta\cdot \sqrt{n_1}(\varphi-\varphi^\star)\right\}\\&+J(\varphi^\star)^{-1}\{\nabla_{\eta\varphi}^2 \mathbb{M}(\eta^\star,\bar\varphi)-\nabla_{\eta\varphi}^2 \mathbb{M}(\eta^\star,\varphi^\star)\}\cdot\vartheta\cdot \sqrt{n_1}(\varphi-\varphi^\star)+o_p(\|\sqrt{n_1}(\varphi-\varphi^\star)\|)
	\end{flalign*}Further, since $\nabla_{\eta\varphi}^2\mathbb{M}(\eta^\star,\varphi)$ is continuous in $\varphi$, and since $\sqrt{n}(\varphi-\varphi^\star)=O_p(1)$ by Lemma 1, we see that 
	\begin{flalign*}
		t_\varphi&=\sqrt{n_2}(\eta-\eta^\star)-J(\varphi^\star)^{-1}\left\{\frac{\Delta_{n_2}(\eta^\star,\varphi^\star)}{\sqrt{n_2}}+\nabla_{\eta\varphi}^2 \mathbb{M}(\eta^\star,\varphi^\star)\cdot\vartheta\cdot \sqrt{n_1}(\varphi-\varphi^\star)\right\}+o_p(1).
	\end{flalign*}
	
	Now, use the fact that $J(\varphi^\star)=\Sigma_{22}$, $\Sigma_{12}=\nabla_{\eta\varphi}^2\mathbb{M}(\theta^\star)$, and re-arrange the first term in $\nu$ as 
	\begin{flalign*}
		t_\varphi&=\sqrt{n_2}(\eta-\eta^\star)-\Sigma_{22}^{-1}\left\{\Delta_{n_2}(\eta^\star,\varphi)/\sqrt{n_2}+\Sigma_{12}\cdot\vartheta\cdot Z_{n_1}\right\}-\Sigma_{22}^{-1}\Sigma_{21}\cdot\vartheta\cdot\{\sqrt{n_1}(\varphi-\varphi^\star)-Z_{n_1}\}+o_p(1)
		\\&=\sqrt{n_2}(\eta-\eta^\star)-Z_{n_2}-\Sigma_{22}^{-1}\Sigma_{21}\cdot\vartheta\cdot\{\sqrt{n_1}(\varphi-\varphi^\star)-Z_{n_1}\}+o_p(1).
	\end{flalign*}
	Recalling $$\xi:=\sqrt{n_2}(\eta-\eta^\star)-Z_{n_2},\quad\phi:=\sqrt{n_1}(\varphi-\varphi^\star)-Z_{n_1},$$ we then see that 
	$$
	t_\varphi=\xi-\vartheta\cdot\Sigma_{12}\phi+o_p(1).
	$$Applying this into equation \eqref{eq:a12} then yields 
	$$
	M_n(\eta,\varphi)-M_n(\eta^\star,\varphi^\star)=-\frac{1}{2}\left(\xi-\vartheta\cdot\Sigma_{12}\phi\right)^{\top}\Sigma_{22}\left(\xi-\vartheta\cdot\Sigma_{12}\phi\right)+\frac{1}{2}Z_{n_2}(\varphi^\star)^{\top}\Sigma_{22}Z_{n_2}(\varphi^\star)+R_n(\eta,\varphi),
	$$
	
	Similarly, from Assumption \ref{ass:ident_L}, we have the following expansion for $L_{n_1}(\varphi)-L_{n_1}(\varphi^\star)$:
	\begin{flalign*}
		L_{n_1}(\varphi)-L_{n_1}(\varphi^\star)&=\sqrt{n_1}(\varphi-\varphi^\star )^{\top}\nabla_\varphi L_{n_1}(\varphi^\star)/\sqrt{n_1}-\frac{n_1}{2}(\varphi-\varphi^\star)^{\top}[-\nabla_{\varphi\varphi}\mathbb{L}(\varphi^\star)](\varphi-\varphi^\star)+R_{4n}(\varphi)\\&=-\frac{1}{2}\{\sqrt{n_1}(\varphi-\varphi^\star)-Z_{n_1}\}^{\top}\Sigma_{11}\{\sqrt{n_1}(\varphi-\varphi^\star)-Z_{n_1}\}+\frac{1}{2}Z_{n_1}^{\top}\Sigma_{11}^{-1}Z_{n_1}	R_{4n}(\varphi)\\&=-\frac{1}{2}\phi^{\top}\Sigma_{11}\phi +\frac{1}{2}Z_{n_1}^{\top}\Sigma_{11}^{-1}Z_{n_1}	R_{4n}(\varphi)
	\end{flalign*}where, by Assumption \ref{ass:ident_L}, the remainder term $R_{4n}(\varphi)$ satisfies $R_{4n}(\varphi)/[1+n_1\|\varphi-\varphi^\star\|^2]=o_p(1)$. 
	
	Recalling that $  Q_n(\theta)=  L_{n_1}(\varphi)+  M_{n_2}(\eta,\varphi)$, and adding the two expansions together yields, for $T_n=D_n^{-1/2}Z_n+\theta^\star$, 
	\begin{flalign*}
		Q_n(D_n^{-1/2}t+T_n) - Q_n(\theta^\star)	&=-\frac{1}{2}\phi^{\top}\Sigma_{11}^{}\phi -\frac{1}{2}\{\xi-\vartheta\cdot \Sigma_{22}^{-1}\Sigma_{21}\phi\}^{\top}\Sigma_{22}\{\xi-\vartheta\cdot \Sigma_{22}^{-1}\Sigma_{21}\phi\}\\&+\frac{1}{2}Z_{n_1}^{\top}\Sigma_{11}^{-1}Z_{n_1}+\frac{1}{2}Z_{n_2}^{\top}\Sigma_{22}^{-1}Z_{n_2}+\sum_{j=1}^{4}R_{jn}(D_n^{-1/2}t+T_n).
	\end{flalign*}
	Lastly, we can rewrite the above equation in the following form: 
	\begin{flalign*}
		Q_n(D_n^{-1/2}t+T_n) - Q_n(\theta^\star)	&=-\frac{1}{2}(\phi^\top,\xi^\top)^{\top}\begin{pmatrix}
			\Sigma_{11}^{}+\vartheta^2\Sigma_{12}\Sigma_{22}^{-1}\Sigma_{21} &\vartheta\cdot \Sigma_{12}\\\vartheta\cdot \Sigma_{21}&\Sigma_{22}
		\end{pmatrix}\begin{pmatrix}
			\phi\\ \xi
		\end{pmatrix}+\frac{1}{2}Z_{n_1}^{\top}\Sigma_{11}^{-1}Z_{n_1}\nonumber\\&+\frac{1}{2}Z_{n_2}^{\top}\Sigma_{22}^{-1}Z_{n_2}+\sum_{j=1}^{4}R_{jn}(D_n^{-1/2}t+T_n).
	\end{flalign*}
	Recalling the definition of $V^{-1}$ given in equation \eqref{eq:eq_v}, for $t=(\phi^{\top},\xi^{\top})^{\top}$ we have that 
	\begin{flalign}\label{eq:newform}
		Q_n(D_n^{-1/2}t+T_n) - Q_n(\theta^\star)	&=-\frac{1}{2}t^\top V^{-1}t+\frac{1}{2}Z_{n_1}^{\top}\Sigma_{11}^{-1}Z_{n_1}+\frac{1}{2}Z_{n_2}^{\top}\Sigma_{22}^{-1}Z_{n_2}\nonumber +\sum_{j=1}^{4}R_{jn}(D_n^{-1/2}t+T_n),
	\end{flalign}
	and the cut posterior $\pi(t\mid \bw,\bz)$ can be restated as 
	\begin{flalign*}
		\pi(t|\bw,\bz)&=\frac{\pi\{D_{n}^{-1/2}t+T_{n}\}\exp[ \{Q_n\{D_{n}^{-1/2}t+T_{n}\}-Q_n(\theta^\star)\}]}{\int_{\mathcal{T}_n} {\pi\{D_{n}^{-1/2}t+T_{n}\}\exp[ \{Q_n\{D_{n}^{-1/2}t+T_{n}\}-Q_n(\theta^\star)\}]\dt t}}=\frac{\pi\{D_{n}^{-1/2}t+T_{n}\}\exp\{\omega(t)\}}{C_n},
	\end{flalign*}where
	$$
	\omega(t) = -\frac{1}{2}t^\top V^{-1}t +\sum_{j=1}^{4}R_{jn}(D_n^{-1/2}t+T_n),
	$$
	and 
	\begin{flalign*}
		C_{n}:=&\int_{\mathcal{T}_n} \pi(D_{n}^{-1/2}t+T_{n})\exp\{Q_n(D_{n}^{-1/2}t+T_{n})-Q_n(\theta^\star)\}\dt t.
	\end{flalign*}
	
	The stated result follows if
	\begin{flalign*}
		\int_{\mathcal{T}_n}  \left|\pi(t|\bw,\bz)-N\{t;0,V\}\right|\dt t&=C_n^{-1}J_n=o_p(1), 
	\end{flalign*}where
	\begin{flalign*}
		J_{n}&=\int_{\mathcal{T}_n} \bigg{|}\exp\left\{\omega(t)\right\} \pi_{}\left\{T_n+D_n^{-1/2}{t_{}}\right\}-C_nN\{t;0,V^{-1}\}\bigg{|} \dt t.
	\end{flalign*}The above equation takes precisely the same form as in the proof of Theorem \ref{prop:bvm1}, but where the remainder term is now $R_{1n}(\theta)+R_{2n}(\theta)+R_{3n}(\theta)+R_{4n}(\varphi)$. Hence, so long as this new remainder satisfies Assumption \ref{ass:expand}(iii), the proof follows the same arguments used in Theorem \ref{prop:bvm1}. A sufficient condition for this new remainder term to satisfy Assumption \ref{ass:expand}(iii) is that Assumption \ref{ass:expand}(iii) is satisfied for each term. We note that $R_{1n}(\theta)$ and $R_{4n}(\varphi)$ both satisfy the condition by hypothesis,  while  Lemma \ref{lem:remain} verifies Assumption \ref{ass:expand}(iii) for $R_{2n}(\theta)$ and $R_{3n}(\theta)$. 
	
	The remainder of the proof follows the same arguments as those used to prove Theorem \ref{prop:bvm1} and is omitted for the sake of brevity. 
\end{proof}

\subsection{Lemmas}

\begin{lemma}\label{lem:remain}
	Under the assumptions of Corollary \ref{corr:bvm_rev}, Assumption 4(iii) is satisfied for $R_{2n}(\theta)$, and $R_{3n}(\theta)$.
\end{lemma}
\begin{proof}
	For $j=2,3$, Assumption 4(iii) is equivalent to the following condition: for any $\delta_n=o(1)$, 
	\begin{equation}\label{eq:ref_rem}
		\sup_{\|\theta-\theta^\star\|\le\delta_n}\frac{|R_{jn}(\theta)|}{1+n\|\theta-\theta^\star\|^2}=o_p(1).
	\end{equation}
	We verify \eqref{eq:ref_rem} separately for $j=2$ and $j=3.$ 
	\medskip
	
	\noindent\textbf{Term $R_{2n}(\theta)$:} Define the matrix function $V(\varphi,\varphi^\star):=\left\{\nabla_{\eta\varphi} \mathbb{M}(\eta^\star,\varphi)-\nabla_{\eta\varphi} \mathbb{M}(\eta^\star,\varphi^\star)\right\}$ and consider
	\begin{flalign*}
		R_{2n}(\theta)&={n}{}(\eta-\eta^\star)^{\top}V(\bar\varphi,\varphi^\star)(\varphi-\varphi^\star)=\frac{1}{2}\sqrt{n}(\theta-\theta^\star)^{\top}\begin{pmatrix}
			0&V(\bar\varphi,\varphi^\star)\\V(\bar\varphi,\varphi^\star)&0
		\end{pmatrix}\sqrt{n}(\theta-\theta^\star),
	\end{flalign*}where $\bar\varphi$ is some intermediate value satisfying $\|\bar\varphi-\varphi^\star\|\le\|\varphi-\varphi^\star\|$. We then see that 
	$$
	|R_{2n}(\theta)|\le \|\sqrt{n}(\theta-\theta^\star)\|^2\|V(\bar\varphi,\varphi^\star)\|^2,
	$$% where we have used the fact that $\|\bar\varphi-\varphi^\star\|\le\|\varphi-\varphi^\star\|$, by definition of the intermediat value $\bar\varphi$. 
	and
	$$
	\sup_{\|\theta-\theta^\star\|\le\delta_n}\frac{|R_{2n}(\theta)|}{1+n\|\theta-\theta^\star\|^2}\le \sup_{\|\theta-\theta^\star\|\le\delta_n}\|V(\bar\varphi,\varphi^\star)\|^2\frac{\|\sqrt{n}(\theta-\theta^\star)\|^2}{1+\|\sqrt{n}(\theta-\theta^\star)\|^2}
	\le \|V(\bar\varphi,\varphi^\star)\|^2,
	$$since $\sup_{\|\theta-\theta^\star\|\le\delta_n}\frac{\|\sqrt{n}(\theta-\theta^\star)\|^2}{1+\|\sqrt{n}(\theta-\theta^\star)\|^2}\le 1$ for any $\delta_n=o(1)$. 
	
	From the definition of the intermediate value $\bar\varphi$, we have that $\|\bar\varphi-\varphi^\star\|\le\|\varphi-\varphi^\star\|\le\delta_n$. From Assumption \ref{ass:fur_exp}, $V(\varphi,\varphi^\star)$ is continuous in $\varphi$ for all $\varphi$ in a neighborhood of $\varphi^\star$. Conclude from this continuity that $\|V(\varphi,\varphi^\star)\|^2=o(1)$ when $\|\theta-\theta^\star\|\le\delta_n$. Equation \eqref{eq:ref_rem} is satisfied for $R_{2n}(\theta)$. 
	
	\medskip
	
	\noindent\textbf{Term $R_{3n}(\theta)$.} Now, let 
	$V(\varphi,\varphi^\star):=[J(\eta^\star,\varphi)-J(\eta^\star,\varphi^\star)]$. Similar to the proof of the $R_{2n}(\theta)$ term, 
	\begin{flalign*}
		R_{3n}(\theta)&=-\frac{n}{2}(\eta-\eta^\star)^{\top}V(\varphi,\varphi^\star)(\eta-\eta^\star)=-\sqrt{n}(\theta-\theta^\star)^{\top}\begin{pmatrix}
			0&0\\0&V(\varphi,\varphi^\star)
		\end{pmatrix}\sqrt{n}(\theta-\theta^\star),
	\end{flalign*}so that 
	$
	|R_{3n}(\theta)|\le \|\sqrt{n}(\theta-\theta^\star)\|^2\|V(\varphi,\varphi^\star)\|^2.
	$ Repeating the same argument as used in the first part of the result, we have that
	\begin{flalign*}
		\sup_{\|\theta-\theta^\star\|\le\delta_n}\frac{|R_{3n}(\theta)|}{1+n\|\theta-\theta^\star\|^2}&\le \sup_{\|\theta-\theta^\star\|\le\delta_n}\|V(\varphi,\varphi^\star)\|^2\frac{\|\sqrt{n}(\theta-\theta^\star)\|^2}{1+\|\sqrt{n}(\theta-\theta^\star)\|^2}\\&\le \sup_{\|\theta-\theta^\star\|\le\delta_n}\|V(\varphi,\varphi^\star)\|^2=o(1).
	\end{flalign*}
	
\end{proof}

\section{Credible sets for understanding uncertainty propagation in the cut posterior}

The following algorithm describes the construction of credible sets for $\eta$
repeatedly for samples of $\varphi$ from the cut posterior, using the large
sample approximation of Theorem 1.  
\begin{algorithm}[H]
	\caption{Diagnostic for $\eta|\bw,\varphi$}
	\label{alg:CS}
	\begin{algorithmic}
		\STATE Inputs: A sequence $\varphi_1,\dots,\varphi_M$, and quantile $\alpha$.  
		\STATE Output: A sequence of approximate confidence sets $\{C^\eta_\alpha(\varphi_j):j\le M\}$
		\FOR{$j=1,\dots,M$ and $\varphi=\varphi_j$}
		\STATE Estimate $\eta^\star_\varphi$, $J(\eta^\star_\varphi|\varphi)$ by $\hat{\eta}_\varphi$, $J_n(\hat{\eta}_\varphi|\varphi)$.
		\STATE Draw: $Z_k\stackrel{iid}{\sim} N\{\hat{\eta}_\varphi,[\nu n J(\hat{\eta}_\varphi|\varphi)]^{-1}\}$, for $k=1,\dots,K$	
		\STATE Calculate $Y_k=[n\nu J_n(\hat{\eta}_\varphi|\varphi)]^{-1/2}\{Z_k-\hat{\eta}_\varphi\}$, 
		\STATE Retain all $Y_k$ such that $\|Y_k\|^2\le \chi^2_{d_\eta}(1-\alpha)$. 
		%		\STATE Return: ${\text{RMSE}}_K(\varphi_j,\bar\varphi)=\sum_{i=1}^{K}(Z_k-\bar\eta)^2/K$
		\ENDFOR	
	\end{algorithmic}	
\end{algorithm}

\section{Marginal semi-modular posterior}

Algorithm \ref{alg:divs} describes how to draw MCMC samples from our
proposed semi-modular posterior density introduced in Section 3.4.  
\begin{algorithm}[H]
	\caption{Semi-modular posterior}
	\label{alg:divs}
	\begin{algorithmic}
		\STATE Inputs: a value of $\gamma$, and a transition kernel $q(\varphi|\varphi')$. 
		\STATE Output: Draws from the approximate semi-modular cut posterior $\widehat\pi^\gamma_\cut(\varphi|\bz,\bw)$.
		\STATE Initialize $\varphi^{(0)}$
		\FOR{$j=1,\dots,M$ and $\varphi=\varphi_j$}
		\STATE Draw $\bar{\varphi}\sim q(\varphi|\varphi^{i-1})$
		\STATE Estimate $\eta_{\bar{\varphi}}$, $J(\eta_{\bar{\varphi}}|\bar{\varphi})$ by $\widehat{\eta}_{\bar{\varphi}}$, $J_{n_2}(\widehat{\eta}_{\bar{\varphi}}|\varphi)$%, $Z_{n_2}(\eta_{\bar{\varphi}}|\bar{\varphi})$
		\STATE Choose $\eta^\star$ in the HPD region of $N\{\eta;\eta_{\bar{\varphi}}%+Z_{n_2}(\eta_{\bar{\varphi}}|\bar{\varphi})/n
		,[n_2\nu J_{n_2}(\widehat{\eta}_{\bar{\varphi}}|\varphi)]^{-1}\}$
		\STATE Compute $\ln\widehat{m}_\eta(\bw|\bar{\varphi})$ via \eqref{eq:margest}.
		\STATE Compute 
		$$L^P=\pi_\cut(\bar{\varphi}|\bz)\exp\{\ln\widehat{m}_\eta(\bw|\bar{\varphi})\}^\gamma,\,\,\,\,L^{i-1}=\pi_{\text{cut}}(\varphi^{i-1}|\bz)\exp\{\ln \widehat{m}_\eta(\bw|\varphi^{i-1})\}^\gamma,$$ and the Metropolis-Hastings ratio:
		$
		r=L^P \pi(\bar{\varphi})q(\varphi^{i-1}|\bar{\varphi})/L^{i-1} \pi(\varphi^{i-1})q(\bar{\varphi}|\varphi^{i-1}).
		$
		\IF{$\mathcal{U}(0,1)<r$}
		\STATE Set $\varphi^i=\bar{\varphi}$	
		\ELSE
		\STATE Set $\varphi^i=\varphi^{i-1}$
		\ENDIF
		\ENDFOR 
		%\STATE Calculate divergence between samples $\varphi^i$ and $\varphi^i_\cut$
	\end{algorithmic}	
\end{algorithm}

Figure \ref{smi} shows semi-modular posterior densities for $\eta_1$
and $\eta_2$ for the method of \cite{carmona2020semi} (top)
and marginal semi-modular
approach of Section 3.3 (bottom) for 
$\gamma\in \{0,0.2,0.4,0.6,0.8,1\}$ for the epidemiological example
of Section 4.1.  We can see that for the same
value of $\gamma$ for the two methods, the semi-modular posterior
densities are similar.  For both approaches, after the SMI samples
for $\varphi$ are drawn, we obtained samples for $\eta$ for each
$\varphi$ sample using the SIR approach described in Section 4.1.  
However, using the normal approximation directly makes little difference
to the result (results not shown).  
\begin{figure}[h]
	\begin{center}
		\begin{tabular}{c}
			\includegraphics[width=85mm]{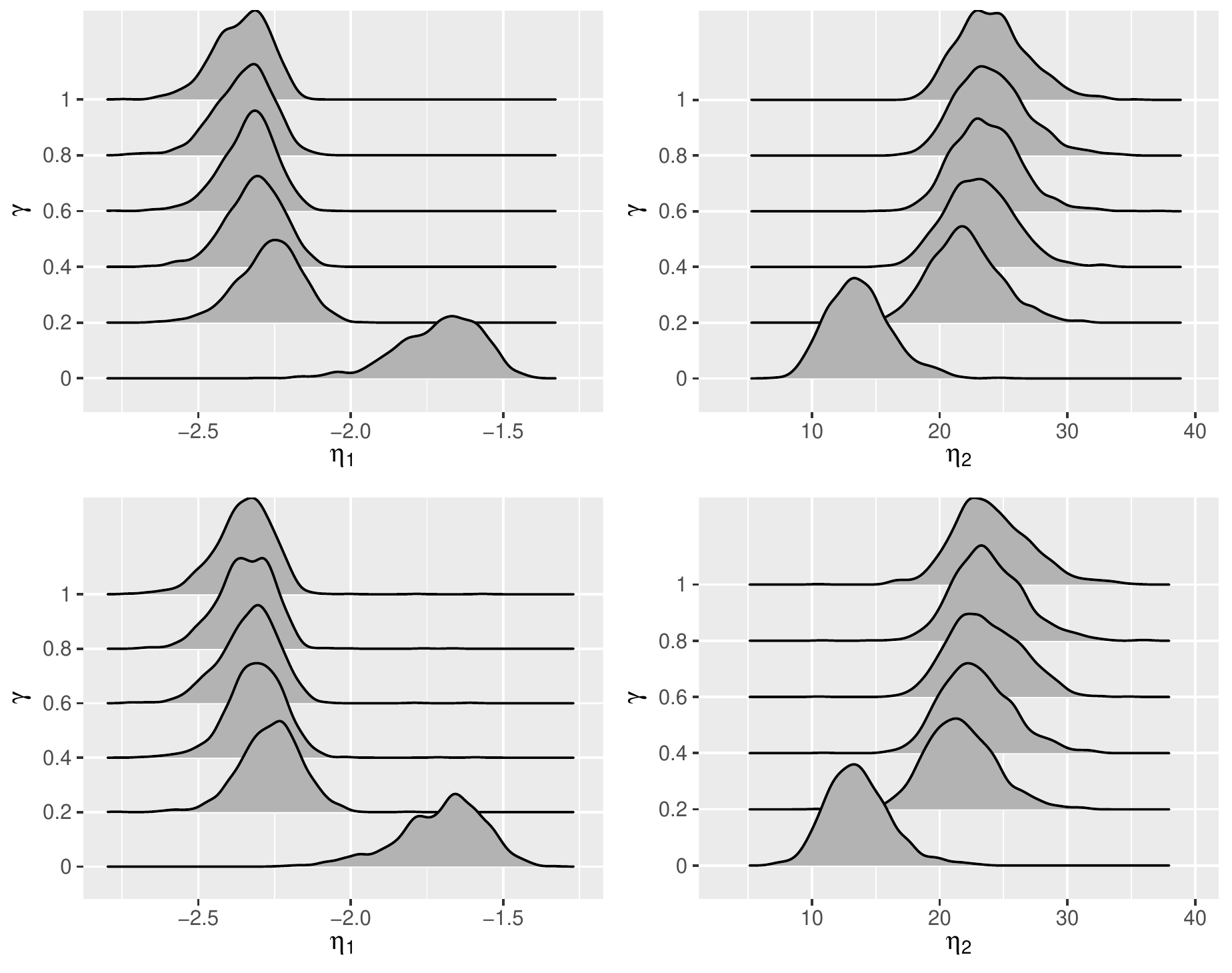} 		
		\end{tabular}
	\end{center}
	\vspace{-.7cm}
	\caption{\label{smi}Semi-modular posterior densities for
		method of \cite{carmona2020semi} (top row) and marginal
		semi-modular method (bottom row) for $\eta_1$ (left) and 
		$\eta_2$ (right) and $\gamma\in \{0,0.2,0.4,0.6,0.8,1\}$}
\end{figure}

%\bibliographystyle{ba}
%\bibliography{mod_bib}

\end{document}